\documentclass{article}
\usepackage{authblk}
\usepackage{fancyhdr}
\usepackage{amscd}
\usepackage[dvipdfmx]{graphicx}
\usepackage{amsmath}
\usepackage{amssymb}
\usepackage{amsthm}
\usepackage{mathrsfs}
\newtheorem{dfn}{Definition}[section]
\newtheorem{thm}[dfn]{Theorem}
\newtheorem{prop}[dfn]{Proposition}
\newtheorem{lem}[dfn]{Lemma}
\newtheorem{cor}[dfn]{Corollary}
\newtheorem{rem}[dfn]{Remark}

\newtheorem{ass}[dfn]{Assumption}

\newtheorem{observation}[dfn]{Observation}

\numberwithin{equation}{section}

\usepackage{color}
\begin{document}

\title{On Algebraic-Dynamical Correspondence of Oscillatory Blow-Up Solutions for ODEs}

%\author{Hisatoshi Kodani\thanks{Advanced Institute for Materials Research, Tohoku University, 2-1-1 Katahira, Aoba-ku, Sendai, 980-8577, Japan ({\tt hisatoshi.kodani@tohoku.ac.jp})}, 
%Kaname Matsue\thanks{(Corresponding author)} $^{,}$ \footnote{Institute of Mathematics for Industry, Kyushu University, Fukuoka 819-0395, Japan ({\tt kmatsue@imi.kyushu-u.ac.jp})} $^{,}$ \footnote{International Institute for Carbon-Neutral Energy Research (WPI-I$^2$CNER), Kyushu University, Fukuoka 819-0395, Japan} ,\\ 
%Hiroyuki Ochiai\thanks{Institute of Mathematics for Industry, Kyushu University, Fukuoka 819-0395, Japan ({\tt ochiai@imi.kyushu-u.ac.jp})}
%and Akitoshi Takayasu\thanks{Faculty of Engineering, Information and Systems, University of Tsukuba, 1-1-1 Tennodai, Tsukuba, Ibaraki 305-8573, Japan (\texttt{takitoshi@risk.tsukuba.ac.jp})}
%}
%\author[1]{Hisatoshi Kodani}
\author[1,2]{Kaname Matsue\thanks{(Corresponding author, {\tt kmatsue@imi.kyushu-u.ac.jp})}}
%\author[1]{\\ Hiroyuki Ochiai}
% and
%\author[3]{Akitoshi Takayasu}

\affil[1]{
\normalsize{
Institute of Mathematics for Industry, Kyushu University, Fukuoka 819-0395, Japan
}
}
\affil[2]{
\normalsize{
International Institute for Carbon-Neutral Energy Research (WPI-I$^2$CNER), Kyushu University, Fukuoka 819-0395, Japan
}
}
%\affil[3]{Faculty of Engineering, Information and Systems, University of Tsukuba, 1-1-1 Tennodai, Tsukuba, Ibaraki 305-8573, Japan}

%\date{}
\maketitle

\begin{abstract}
We develop an algebraic-dynamical correspondence for type-I oscillatory blow-up solutions of autonomous ODEs with asymptotically quasi-homogeneous vector fields. 
Periodic solutions of the balance law arising from asymptotic expansions of blow-ups are related to periodic orbits on the horizon for the desingularized vector field. 
Unlike the stationary case, the correspondence involves a positive periodic scaling function and a nontrivial time reparametrization. 
We further establish the correspondence of linearized structures: distinguished phase and scaling directions generate invariant subbundles, while the remaining characteristic multipliers are identified through associated quotient bundles. 
These results yield a criterion for the existence of type-I periodic blow-up solutions based on periodic solutions of the balance law and their characteristic multipliers, without explicitly constructing compactifications or periodic orbits at infinity.
The theory also clarifies the difference between stability information in the balance law and in dynamics at infinity. 
Two examples illustrate the correspondence.
\end{abstract}


\tableofcontents

\section{Introduction}
\label{section-intro}

Our main interest in the present study is the nature of {\em blow-up solutions} of ODEs depending on the real-time variable $t$:
\begin{equation}
\label{ODE-original}
{\bf y}' \equiv \frac{d{\bf y}}{dt} = f({\bf y}),\quad {\bf y}(0) = {\bf y}_0\in U,
\end{equation}
where $U\subset \mathbb{R}^n$ is an open set, $f: U\to \mathbb{R}^n$ is $C^r$ with $r\geq 1$.
\par
A solution ${\bf y}(t)$ is said to {\em blow up} at $t_{\max} < \infty$ if its modulus (or norm) diverges as $t\to t_{\max}-0$.
The finite value $t_{\max}$ is known as the {\em blow-up time}.
Properties of blow-up solutions such as the existence and asymptotic behavior are of great importance in the field of (ordinary, partial, delay, etc.) differential equations, and have been widely investigated over the decades in a variety of problems, such as {\em ignition in combustion problem, chemotaxis, superconductivity, self-focusing of optical laser, wild oscillations in nonlinear beam problems, finite-time collapse of crystalline curvature flows}. 
See, for example, \cite{asym2} for the references.
Asymptotic behavior of blow-up solutions is often referred to as determination of {\em blow-up rates}, which is characterized as the following form:
\begin{equation*}
{\bf y}(t) \sim {\bf u}(\theta(t))\quad \text{ as }\quad t\to t_{\max},
\end{equation*}
for a function ${\bf u}$, where 
\begin{equation*}
\theta(t) = t_{\max}-t
\end{equation*}
and $t_{\max}$ is assumed to be a finite value and to be known a priori.
%%%%
\par
The author and his collaborators have recently developed a framework to characterize blow-up solutions from the viewpoint of {\em dynamics at infinity} (e.g. \cite{Mat2018, Mat2019}), as well as machineries of computer-assisted proofs for the existence of blow-up solutions extracting both their qualitative and quantitative features (e.g. \cite{LMT2023, MT2020_1, MT2020_2, TMSTMO2017}).
As in the present paper, finite-dimensional vector fields with scale invariance in an asymptotic sense, {\em asymptotic quasi-homogeneity} defined precisely in Definition \ref{dfn-AQH}, are mainly concerned.
In this framework, dynamics at infinity is appropriately characterized, and blow-up solutions are shown to be characterized by dynamical properties of invariant sets, such as equilibria, \lq\lq at infinity".
More precisely, the original vector field is transformed by appropriate change of coordinates with respect to the asymptotic quasi-homogeneity, and resulting vector field is called the {\em desingularized vector field}, and \lq\lq dynamics at infinity" is appropriately described by dynamics on the {\em horizon}, the image of infinity under appropriate transformations referred to as {\em compactifications} in many studies (e.g., \cite{Mat2018}) or {\em embeddings} (e.g., \cite{Mat2025_NHIM}), for the desingularized vector fields.
%%%%
In particular, {\em hyperbolicity} of such invariant sets induces blow-up rates of the form $a_0(t_{\max} - t)^{-\rho}$ uniquely determined by the asymptotic quasi-homogeneity of the original vector field.
By means of the terminology in the field of (partial) differential equations, such blow-ups are said to be {\em type-I}.
\par
On the other hand, {\em multi-order asymptotic expansions as well as their correspondence to several objects in dynamics at infinity} have recently been provided by the author and his collaborators \cite{asym1, asym2}.
While this was originally motivated to derive a systematic methodology to calculate multi-order asymptotic expansions of blow-up solutions, quantities characterizing asymptotic expansions of blow-ups have turned out to be in one-to-one correspondence with (linear) dynamical structure of \lq\lq equilibria at infinity".
Similar correspondence is also provided in nonautonomous systems \cite{M2026_asym_nonaut}.
\par
\bigskip
The above investigations are mainly concerned with blow-up solutions exhibiting essentially monotone divergence. 
On the other hand, there are also non-monotone blow-ups, in particular oscillatory blow-ups (e.g., \cite{BFGK2011, GP2011, GP2013}), as well as unbounded solutions accompanied by oscillations (e.g., \cite{DH1999}). 
Such solutions involve an additional difficulty that is absent in stationary blow-ups. 
Their existence and oscillatory nature are often characterized in terms of upper and lower envelopes, or equivalently by {\rm limsup}/{\rm liminf}-type behavior (e.g., \cite{GP2011}). 
While such characterizations are effective for detecting divergence with oscillations, they provide rather limited information on the concrete asymptotic profile and, in particular, on the dynamical stability underlying the blow-up behavior.
One possible approach to extracting more detailed information is to introduce an asymptotic ansatz involving an oscillatory profile and to transform the original problem into an equation for a bounded function. 
Periodic or bounded solutions of the transformed equation can then describe oscillatory blow-up profiles; such an approach has been successfully applied, for example, to nonlinear fourth-order equations \cite{CLM2018, DLP2015}. 
Nevertheless, this procedure presupposes an appropriate scaling, a transformed time variable, and the existence of a bounded profile of the prescribed form. 
These ingredients are not always evident a priori, and the applicability and dynamical justification of the ansatz therefore have to be examined separately.
From the viewpoint of dynamics at infinity, on the other hand, oscillatory blow-ups can instead be described through periodic or more general invariant sets on the horizon \cite{Mat2018, Mat2025_NHIM}. 
In principle, this viewpoint is sufficient for describing both the asymptotic behavior and the associated stability. 
In practice, however, a direct investigation of dynamics at infinity may become considerably involved. 
The directions of divergence are generally anisotropic, and the resulting compactified and desingularized vector fields may have complicated expressions. 
Consequently, even when the underlying dynamical mechanism is simple, direct calculations can become lengthy and tedious, and may obscure rather than reveal the essential structure. 
This difficulty becomes particularly serious when the existence of a relevant periodic orbit at infinity is not known in advance: there are only limited systematic tools for efficiently detecting such an object before carrying out the full compactification and dynamical analysis.
\par
The aim of the present paper is to overcome these difficulties by extending the algebraic-dynamical correspondence developed for stationary blow-ups \cite{asym2, M2026_asym_nonaut} to the periodic setting. 
We establish a correspondence between periodic solutions of the balance law arising from blow-up asymptotics and periodic orbits on the horizon for the desingularized vector field. 
Unlike the stationary case, however, both the balance law and the associated variational equations are {\em genuinely time-dependent}. 
In particular, the scaling connecting the two descriptions is no longer a constant but a positive periodic function, accompanied by a nontrivial reparametrization of time. 
Moreover, the relevant linearized objects are no longer eigenspaces and eigenvalues at equilibria, but invariant vector bundles and characteristic multipliers along periodic orbits. 
These time-dependent structures constitute the principal difficulties in extending the stationary correspondence.
The resulting correspondence reveals features that are largely hidden in the stationary setting. 
First, the scaling function connecting the algebraic blow-up profile with the periodic dynamics at infinity is itself characterized by a scalar {\em nonautonomous} linear differential equation. 
Second, the linear invariant structures of the two descriptions are related through appropriate scaling and projection, and their remaining characteristic multipliers are naturally identified on quotient bundles after the distinguished phase and scaling directions are removed. 
Therefore, replacing equilibria by periodic orbits is not merely a formal extension of the stationary theory: it exposes the roles of the scaling function, time reparametrization, and invariant bundle structure as intrinsic ingredients of the algebraic-dynamical correspondence.
As a consequence, the existence of type-I oscillatory blow-up solutions can be investigated directly through periodic solutions of the balance law and their characteristic multipliers, without explicitly constructing a compactification or directly searching for periodic orbits at infinity. 
In this sense, the balance law provides an algebraically simpler representation of the same asymptotic mechanism, while the correspondence with dynamics at infinity supplies its dynamical justification and stability information. 
We expect that this viewpoint not only provides an efficient criterion for periodic blow-ups, but also helps clarify the structural essence of blow-up solutions and of dynamics at infinity beyond the stationary setting.
In particular, the periodic setting makes explicit structures that are reduced to constant quantities and fixed linear spaces in the stationary case.
%%%%
\par
\bigskip
The rest of the present paper is organized as follows. 
In Section \ref{section-preliminary}, we briefly review the description of blow-up solutions for autonomous ODEs through dynamics at infinity. 
In Section \ref{section-correspondence-1}, we establish the correspondence between periodic solutions of the balance law and periodic orbits on the horizon. 
Section \ref{section-correspondence-2} develops the correspondence of invariant subbundles and characteristic multipliers, including the associated quotient-bundle structure. 
In Section \ref{section-example}, we derive a simple criterion for the existence of periodic blow-ups and illustrate the theory with two examples.

\section{Blow-up description for autonomous systems through dynamics at infinity: Short review}
\label{section-preliminary}

Here we briefly review a characterization of blow-up solutions for autonomous, finite-dimensional ODEs from the viewpoint of dynamical systems.
Details of the present methodology are already provided in \cite{Mat2018, Mat2025_NHIM}.
%\par
%\bigskip
%
%	new subsection
%
\subsection{Asymptotically quasi-homogeneous vector fields}
\label{sec:QH}

First we describe a basic class of ODEs considered throughout this paper.
\begin{dfn}[Homogeneity index and admissible domain. cf. \cite{Mat2018, Mat2025_NHIM}]\rm
\label{dfn-index}
Let $\alpha = (\alpha_1,\cdots, \alpha_n)$ be a set of nonnegative integers.
We call the index set $I_\alpha=\{i\in \{1,\cdots, n\}\mid \alpha_i > 0\}$ the set of {\em homogeneity indices associated with $\alpha = (\alpha_1,\cdots, \alpha_n)$}.
Let $U\subset \mathbb{R}^n$.
We say that the domain $U\subset \mathbb{R}^n$ is {\em admissible with respect to the sequence $\alpha$}
if
\begin{equation*}
U = \left\{{\bf u}=(u_1,\cdots, u_n)\in \mathbb{R}^n \mid u_i\in \mathbb{R}\text{ if }i\in I_\alpha,\ (u_{j_1},\cdots, u_{j_{n-l}}) \in \tilde U\right\},
\end{equation*} 
where $\{j_1, \cdots, j_{n-l}\} = \{1,\cdots, n\}\setminus I_\alpha$ and $\tilde U$ is an open set in $\mathbb{R}^{n-l}$ spanning $(u_{j_1},\cdots, u_{j_{n-l}})$ with $\{j_1, \cdots, j_{n-l}\} = \{1,\cdots, n\}\setminus I_\alpha$.
\end{dfn}

\begin{dfn}[Asymptotically quasi-homogeneous vector fields, cf. \cite{asym1, D1993, asym2, Mat2018, Mat2025_NHIM}]\rm
\label{dfn-AQH}
Let $U\subset \mathbb{R}^n$ be an admissible set with respect to $\alpha$. 
Also, let $f_0:U \to \mathbb{R}$ be a function.
Let $\alpha_1,\ldots, \alpha_n$ be nonnegative integers with $(\alpha_1,\ldots, \alpha_n) \not = (0,\ldots, 0)$ and $k > 0$.
We say that $f_0$ is a {\em quasi-homogeneous function\footnote{
%	footnote
In preceding studies, all $\alpha_i$'s and $k$ are typically assumed to be natural numbers.
In the present study, on the other hand, the above generalization is valid.
%	footnote: end
} of type $\alpha = (\alpha_1,\ldots, \alpha_n)$ and order $k$}, denoted by $f_0\in QH_{n,1}(\alpha, k)$, if
\begin{equation*}
f_0( s^{\Lambda_\alpha}{\bf x} ) = s^k f_0( {\bf x} )\quad \text{ for all } {\bf x} = (x_1,\ldots, x_n)^T \in U \text{ and } s>0,
\end{equation*}
where\footnote{
Throughout the rest of this paper, the power of real positive numbers or functions to matrices is described in a similar manner.
%	footnote
}
\begin{equation*}
\Lambda_\alpha =  {\rm diag}\left(\alpha_1,\ldots, \alpha_n\right),\quad s^{\Lambda_\alpha}{\bf x} = (s^{\alpha_1}x_1,\ldots, s^{\alpha_n}x_n)^T.
\end{equation*}
Next, let $X = \sum_{i=1}^n f_i({\bf x})\frac{\partial }{\partial x_i}$ be a continuous vector field defined on $U$.
We say that $X$, or simply $f = (f_1,\ldots, f_n)^T$ is a {\em quasi-homogeneous vector field of type $\alpha = (\alpha_1,\ldots, \alpha_n)$ and order $k+1$}, denoted by $f\in QH_n(\alpha, k)$, if $f_i\in QH_{n,1}(\alpha, k+\alpha_i)$ for all $i=1,\ldots, n$.
%each component $f_i$ is a quasi-homogeneous function of type $\alpha$ and order $k + \alpha_i$.
\par
Finally, we say that $X = \sum_{i=1}^n f_i({\bf x})\frac{\partial }{\partial x_i}$, or simply $f: U\to \mathbb{R}^n$ is an {\em asymptotically quasi-homogeneous vector field of type $\alpha = (\alpha_1,\ldots, \alpha_n)$ and order $k+1$ (at infinity)}, denoted by $f\in AQH_n(\alpha, k)$, if there is
% a quasi-homogeneous vector field  $f_{\alpha,k} = (f_{i; \alpha,k})_{i=1}^n$ of type $\alpha$ and order $k+1$ 
$f_{\alpha, k}\in QH_n(\alpha, k)$
such that
\begin{equation}
\label{residual}
f_i( s^{\Lambda_\alpha}{\bf x} ) - s^{k+\alpha_i} f_{i;\alpha,k}( {\bf x} ) = o(s^{k+\alpha_i}),\quad i\in \{1,\ldots, n\}
 \end{equation}
as $s\to +\infty$ uniformly on $\left\{{\bf x}\in U \mid \sum_{i\in I_\alpha} x_i^2 = 1, (x_{j_1},\cdots, x_{j_{n-l}}) \in \tilde K\right\}$ for any compact subset $\tilde K\subset \tilde U$.
\end{dfn}

%\begin{rem}
%In the above definition, non-polynomial-like functions such as $\sin x$ are not included to characterize quasi-homogeneity.
%Indeed, such functions are allowed to exist only in the residual terms characterized by the asymptotic quasi-homogeneity (\ref{residual}). 
%On the other hand, (\ref{residual}) is required for all $i\in \{1,\ldots, n\}$.
%\end{rem}

%Throughout the rest of this section, consider an (autonomous) $C^r$ vector field (\ref{ODE-original}) with $r\geq 1$, where $f: U \to \mathbb{R}^n$ is asymptotically quasi-homogeneous of type $\alpha = (\alpha_1,\ldots, \alpha_n)$ and order $k+1$ at infinity defined on an admissible set $U\subset \mathbb{R}^n$ with respect to $\alpha$.

A fundamental property of quasi-homogeneous functions and vector fields is reviewed in e.g. \cite{asym1, asym2}, while some fundamental property of quasi-homogeneous functions and vector fields are reviewed here.
%\begin{lem}[\cite{asym2}]
%\label{temporary-label2}
%A quasi-homogenous function $f_0$ of type $(\alpha_1,\ldots,\alpha_n)$ and order $k$ satisfies the following differential equation:
%\begin{equation}\label{temporary-label1}
%\sum_{l=1}^n \alpha_l y_l \frac{\partial f_0}{\partial y_l}({\bf y}) = k f_0({\bf y}),
%\end{equation}
%equivalently
%\begin{equation*}
%(\nabla_{\bf y} f_0({\bf y}))^T \Lambda_{\alpha} {\bf y} = k f_0({\bf y}).
%\end{equation*}
%\end{lem}
See \cite{asym2} for detailed arguments about asymptotic behavior of their derivatives.

\begin{lem}[\cite{asym2}]
\label{lem-identity-QHvf}
%A quasi-homogeneous vector field $f=(f_1,\ldots,f_n)$ of type $\alpha = (\alpha_1,\ldots,\alpha_n)$ and order $k+1$ 
A function $f\in QH_n(\alpha, k)$
satisfies\footnote{
%	footnote
As mentioned in \cite{asym2}, these characterizations still hold when the type $\alpha$ contains a component with $\alpha_i = 0$.
In particular, the framework involving quasi-homogeneity discussed here can be applied to dynamics at infinity for a parameter-dependent family of autonomous ODEs.
%	footnote: end
} 
\begin{equation}\label{temporay-label4}
(D_{\bf y} f)({\bf y}) \Lambda_\alpha \mathbf{y} = \left( k I+ \Lambda_\alpha \right) f({\bf y}).
\end{equation}
\end{lem}

Now consider the autonomous $C^r$ vector field
\begin{equation}
\label{ODE-main}
{\bf y}' = f({\bf y})
\end{equation}
with $r\geq 1$, 
where $f: U(\subset \mathbb{R}^n)\to \mathbb{R}^n$ is $f\in AQH_n(\alpha, k)$ with $\alpha = (\alpha_1,\ldots, \alpha_n)$, 
where $U\subset \mathbb{R}^n$ is an admissible set with respect to $\alpha$.
In particular, the component-wise expression of $f$ is described as follows:
\begin{equation}
\label{f-param}
f({\bf y}) = (f_1({\bf y}), \ldots, f_n({\bf y})).
\end{equation}
Throughout the rest of our paper, we assume that the vector field $f$ has the form (\ref{f-param}).
%%%%

\subsection{Quasi-parabolic embeddings}
\label{sec:global}

Here we review an example of embeddings which embed the original (locally compact) phase space into a compact manifold to characterize \lq\lq infinity" as a bounded object.
While there are several choices of embeddings, the following embedding to characterize {\em dynamics at infinity} is applied here.

\begin{dfn}[Quasi-parabolic embedding, cf. \cite{MT2020_1}]\rm
\label{dfn-quasi-para}
Let $\{\beta_i\}_{i\in I_\alpha}$ be the collection of natural numbers so that 
\begin{equation}
\label{LCM}
\alpha_i \beta_i \equiv c \in \mathbb{N},\quad i\in I_\alpha
\end{equation}
is the least common multiple.
In particular, $\{\beta_i\}_{i\in I_\alpha}$ is chosen to be the smallest among possible collections.
If $i\not \in I_\alpha$, we formally set $\beta_i = 0$.
\par
Let $p_\alpha({\bf y})$ be a functional given by
\begin{equation}
\label{func-p}
p_\alpha({\bf y}) \equiv \left( \sum_{i \in I_\alpha} y_i^{2\beta_i} \right)^{1/2c}.
\end{equation}
Define the mapping $T_{{\rm para};\alpha}: \mathbb{R}^n \to \mathbb{R}^n$ as the inverse of
\begin{equation}
\label{parabolic-cpt}
S_{{\rm para};\alpha}({\bf x}) = {\bf y},\quad y_j = \kappa_\alpha^{\alpha_j} x_j,\quad j=1,\ldots, n,
\end{equation}
where 
\begin{equation*}
\kappa_\alpha = \kappa_\alpha({\bf x}) = (1- p_\alpha({\bf x})^{2c})^{-1} \equiv \left( 1 - \sum_{j\in I_\alpha} x_j^{2\beta_j}\right)^{-1}.
\end{equation*}
We call the mapping $T_{{\rm para};\alpha}$ the {\em quasi-parabolic embedding (with type $\alpha$)}.
\end{dfn}

\begin{rem}
\label{rem-kappa}
The functional $\kappa_\alpha = \tilde \kappa_\alpha({\bf y})$ as a functional determined by ${\bf y}$ is implicitly determined by $p_\alpha({\bf y})$.
Details of such a characterization of $\kappa_\alpha$ in terms of ${\bf y}$, and the bijectivity and smoothness of $T_{{\rm para};\alpha}$ are shown in \cite{MT2020_1} with a general class of embeddings, where the embedding is referred to as {\em compactifications}.
\end{rem}

As proved in \cite{MT2020_1}, $T_{{\rm para};\alpha}$ maps $U$ one-to-one onto the set
$\mathcal{D} \equiv \{{\bf x}\in U \mid p_\alpha({\bf x}) < 1\}$.
Infinity in the original coordinate then corresponds to a point on the level set of $p_\alpha$:
\begin{equation*}
\mathcal{E} = \{{\bf x} \in U \mid p_\alpha({\bf x}) = 1\}.
\end{equation*}
\begin{dfn}\rm
We call the set $\mathcal{E}$ the {\em horizon}.
\end{dfn}

\subsection{Dynamics at infinity and blow-up description}
\label{label-dynamics-infinity}
Once we fix an embedding associated with the type $\alpha = (\alpha_1, \ldots, \alpha_n)$ of the vector field $f\in AQH_n(\alpha, k)$ of the form (\ref{f-param}), %we can derive the vector field which makes sense including the horizon.
the {\em dynamics at infinity} makes sense through the appropriately transformed vector field called the {\em desingularized vector field}, denoted by $g$.
The common approach is twofold.
Firstly, we rewrite the vector field (\ref{ODE-original}) with respect to the new variable used in embeddings.
Secondly, we introduce the time-scale transformation of the form $d\tau = q({\bf x})\kappa_\alpha({\bf x}(t))^k dt$ for some function $q({\bf x})$ which is bounded including the horizon. 
We then obtain the vector field with respect to the new time variable $\tau$, which is {\em as smooth as $f$ including the horizon}, as mentioned in preceding papers (e.g., \cite{Mat2018, MT2020_1}).

%\begin{rem}
%\label{rem-choice-cpt}
%Continuity of the desingularized vector field $g$ including the horizon is guaranteed by the smoothness of $f$ and asymptotic quasi-homogeneity (\cite{MT2020_1}).
%In the case of parabolic-type embeddings, $g$ inherits the smoothness of $f$ including the horizon, which is not always the case of other embeddings in general. 
%Details are discussed in \cite{Mat2018}.
%\end{rem}

\begin{dfn}[Time-scale desingularization]\rm 
Define the new time variable $\tau$ by
\begin{equation}
\label{time-desing-nonaut}
d\tau = (1-p_\alpha({\bf x})^{2c})^{-k}\left\{1-\frac{2c-1}{2c}(1-p_\alpha({\bf x})^{2c}) \right\}^{-1}dt,
\end{equation}
equivalently
\begin{equation*}
t - t_0 = \int_{\tau_0}^\tau \left\{1-\frac{2c-1}{2c}(1-p_\alpha({\bf x}(\tau))^{2c}) \right\}(1-p_\alpha({\bf x}(\tau))^{2c})^k d\tau,
\end{equation*}
where $t_0$ and $\tau_0$ denote the correspondence of initial times, ${\bf x}(\tau) = T({\bf y}(t(\tau)))$ and ${\bf y}(t(\tau))$ is a solution ${\bf y}(t)$ under the parameter $\tau$.
We shall call (\ref{time-desing-nonaut}) {\em the time-scale desingularization of order $k+1$}.
\end{dfn}
%%%
%%%

We then obtain the corresponding desingularized vector field $g$;
\begin{align}
\label{desing-para}
\frac{d{\bf x}}{d\tau} 
%%%%
	&= \left(1-\frac{2c-1}{2c}(1-p_\alpha({\bf x})^{2c}) \right)\tilde f({\bf x}) - G({\bf x})\Lambda_\alpha {\bf x}
\end{align}
with 
\begin{align}
\notag
\tilde f({\bf x}) &:= \kappa_\alpha^{- (k I +\Lambda_\alpha)} f(\kappa_\alpha^{\Lambda_\alpha} {\bf x}),\\
\label{Gx}
G({\bf x}) &\equiv \sum_{j\in I_\alpha} \frac{x_j^{2\beta_j-1}}{\alpha_j}\tilde f_j({\bf x}),\quad 
\Lambda_\alpha = {\rm diag}(\alpha_1, \ldots, \alpha_n).
\end{align}
%where we have used an alias $x_0 = t$.
%The above identifications are also consistent with the evolution of $t$ followed by the time-scale desingularization (\ref{time-desing-nonaut}).

Smoothness of $f$ and the asymptotic quasi-homogeneity guarantee the smoothness of the right-hand side $g$ of (\ref{desing-para}) including the horizon $\mathcal{E}\equiv \{p_\alpha({\bf x}) = 1\}$. %, {\em provided\footnote{
%%	footnote
%This requirement is essential in the nonautonomous case, while it is not the case of the autonomous case.
%%	footnote: end
%}
% $k\geq 1$}.
In particular, {\em dynamics at infinity}, such as divergence of solution trajectories to specific directions, is characterized through dynamics generated by (\ref{desing-para}) around the horizon. 
See e.g., \cite{Mat2025_NHIM, MT2020_1} for details.

\begin{rem}[Invariant structure]
\label{rem-invariance}
The horizon $\mathcal{E}$ is a codimension one invariant submanifold of $\overline{\mathcal{D}}$. 
Indeed, direct calculations yield that
\begin{equation*}
\left. \frac{d}{d \tau}p_\alpha({\bf x}(\tau))^{2c}\right|_{\tau=0} = 0\quad \text{ whenever }\quad {\bf x}(0)\in \mathcal{E}.
\end{equation*}
\end{rem}

\par
\bigskip
Through the embedding we have introduced, 
dynamics around the horizon characterize dynamics at infinity, including blow-up behavior.
For a set $S_\ast$, let
% for a desingularized vector field $g$, let 
\begin{equation*}
W_{\rm loc}^s(S_\ast) = W_{\rm loc}^s(S_\ast; g) := \{ {\bf x} \in U \mid {\rm dist} (\varphi_{g}(\tau, {\bf x}), S) \to 0\, \text{ as }\,\tau \to +\infty\}
\end{equation*}
be the {\em (local) stable set} of ${\bf x}_\ast$ for the dynamical system generated by $g$, where $U$ is a neighborhood of $S_\ast$ in $\mathbb{R}^n$ or an appropriate phase space, $\varphi_{g}$ is the flow generated by $g$, and ${\rm dist}$ denotes the distance associated with the norm $|\cdot|$:
\begin{equation*}
{\rm dist} ({\bf x}, S) = {\rm inf}\{|{\bf x} - {\bf x}_s| \mid {\bf x}_s\in S\}.
\end{equation*}

Let $P(g)$ be a set of periodic orbits for $g$, and $P_{\rm hyp}(g)$ be a set of hyperbolic periodic orbits for $g$. 
Finally, let
\begin{equation*}
P_{\mathcal{E}}(g) := \{\gamma \in P(g) \mid \gamma \subset \mathcal{E}\},\quad P_{{\rm hyp}, \mathcal{E}}(g) := P_{\rm hyp}(g) \cap P_{\mathcal{E}}(g).
\end{equation*}

\begin{thm}[Type-I periodic blow-up, \cite{Mat2018}]
\label{thm:blowup}
Let ${\bf y}(t)$ be a solution of (\ref{ODE-original}) with $f\in AQH_n(\alpha, k)$. 
Suppose that $\gamma_\ast \in P_{{\rm hyp}, \mathcal{E}}(g)$.
If there is a solution ${\bf y}(t)$ of (\ref{ODE-original}) with a bounded initial point ${\bf y}(0)$ whose image ${\bf x} = T_{\rm para}({\bf y})$ is on the local stable manifold $W^s_{\rm loc}(\gamma_\ast; g)$, then $t_{\max} < \infty$ holds; namely, ${\bf y}(t)$ is a blow-up solution.
Moreover,
\begin{equation*}
p_\alpha ({\bf y}(t)) \sim c' \theta(t)^{-1/k}\quad \text{ as }\quad t\to t_{\max}
\end{equation*}
where $c' > 0$ is a constant.
Finally, if the $j$-th component ${\bf x}_\ast$ is not identically zero, then we also have
\begin{equation*}
y_j(t) \sim c_j(-c' \ln \theta(t))\theta(t)^{-\alpha_j /k}\quad \text{ as }\quad t\to t_{\max},
\end{equation*}
where $c_j$ is a periodic function with a constant $c'$.
\end{thm}

%%%

\par
\bigskip

We end this section by providing several properties to describe features of dynamics on the horizon. 
First we consider a function $G$ in (\ref{Gx}), which is essential to characterize equilibria on the horizon mentioned in Section \ref{section-review-asym}.
The gradient of the horizon $\mathcal{E} = \{p_\alpha({\bf x}) = 1\}$ at ${\bf x}\in \mathcal{E}$ is given by
\begin{equation*}
\nabla_{\bf x} p_\alpha({\bf x}) = \frac{p_\alpha({\bf x})^{1-2c}}{c} \left(\beta_1x_1^{2\beta_1-1}, \ldots, \beta_nx_n^{2\beta_n-1}\right)^T = \frac{1}{c} \left(\beta_1x_1^{2\beta_1-1}, \ldots, \beta_nx_n^{2\beta_n-1} \right)^T,
\end{equation*}
where $\beta_l$ is set as $0$ whenever $\alpha_l = 0$.
In particular, 
\begin{equation}
\label{grad-xast}
\nabla_{\bf x} p_\alpha({\bf x}_\ast) = \frac{1}{c}\left(\beta_1x_{\ast,1}^{2\beta_1-1}, \ldots, \beta_n x_{\ast,n}^{2\beta_n-1} \right)^T
\end{equation}
holds at a point ${\bf x}_\ast \in \mathcal{E}$.
%%%
Similarly, we observe that
\begin{equation}
\label{grad-2c-xast}
\nabla_{\bf x} (p_\alpha( {\bf x})^{2c})^T = 2\left(\beta_1x_1^{2\beta_1-1}, \ldots, \beta_nx_n^{2\beta_n-1} \right)
\end{equation}
for any ${\bf x} \in \mathcal{D}$.
In particular, the $j$-th component with $j\not \in I_\alpha$ is $0$.
Using the gradient, the function $G({\bf x})$ in (\ref{Gx}) is also written by
\begin{equation*}
G({\bf x}) = \sum_{j\in I_\alpha} \frac{\beta_j}{c}x_j^{2\beta_j-1}\tilde f_j({\bf x}) = \frac{1}{2c}\nabla_{\bf x} (p_\alpha({\bf x})^{2c})^T \tilde f({\bf x}). 
\end{equation*}
From (\ref{grad-xast}) and (\ref{grad-2c-xast}), we have $\nabla_{\bf x} (p_\alpha( {\bf x}_\ast)^{2c})^T = 2c \nabla_{\bf x} p_\alpha( {\bf x}_\ast)^T$ for any ${\bf x}_\ast \in \mathcal{E}$, in particular
\begin{equation*}
G({\bf x}_\ast) = \frac{1}{2c}\nabla_{\bf x} (p_\alpha( {\bf x}_\ast)^{2c})^T \tilde f({\bf x}_\ast) = \nabla_{\bf x}p_\alpha( {\bf x}_\ast)^T \tilde f({\bf x}_\ast). 
\end{equation*}

\begin{lem}[cf. \cite{asym2}, Section 2]
\label{lem-G}
\begin{equation*}
\kappa_\alpha^{-1} \frac{d\kappa_\alpha}{d\tau} = G({\bf x}(\tau)),
\end{equation*}
where $G({\bf x})$ is given in (\ref{Gx}).
\end{lem}

\section{Correspondence of periodic solutions}
\label{section-correspondence-1}

This and the next sections address several structural correspondences between dynamics around periodic orbits on the horizon for desingularized vector fields and another nature of blow-ups; {\em asymptotic expansions} derived in \cite{asym1, asym2}.
%%%
As a result, we see that a criterion for their existence is simplified through the latter nature.
Unless otherwise mentioned, let $g$ be the desingularized vector field (\ref{desing-para}) associated with $f$.

\subsection{Short Summary: Balance Law, Dynamics at Infinity and Associated Eigenstructures}
\label{section-review-asym}

Our arguments are based on the correspondence of stationary blow-ups; blow-up solutions associated with equilibria on the horizon \cite{asym2}.
Based on this, we quickly introduce the methodology of multi-order asymptotic expansions of type-I blow-up solutions proposed in \cite{asym1}.
%The method begins with the following ansatz, which can be easily verified through the desingularized vector field and our blow-up description; Theorem \ref{thm:blowup}.
First decompose the vector field $f$ into two terms as follows:
\begin{equation*}
f({\bf y}) = f_{\alpha, k}({\bf y}) + f_{\rm res}({\bf y}),
\end{equation*}
where $f_{\alpha, k}\in QH_n(\alpha, k)$ and $f_{\rm res}$ is the residual (i.e., lower-order) terms.
The componentwise expressions are 
\begin{equation*}
f_{\alpha, k}({\bf y}) = (f_{1;\alpha, k}({\bf y}), \ldots, f_{n;\alpha, k}({\bf y}))^T,\quad f_{\rm res}({\bf y}) = (f_{1;{\rm res}}( {\bf y}), \ldots, f_{n;{\rm res}}({\bf y}))^T,
\end{equation*}
respectively.
Similar expressions are also used for other vector fields such as $f$.
%%%

\begin{dfn}[Balance law]\rm
\label{dfn-balance}
We call the identity
\begin{equation}
\label{eq-balance}
\frac{d{\bf Y}_0}{d\eta} = -\frac{1}{k}\Lambda_\alpha {\bf Y}_0 + f_{\alpha, k}({\bf Y}_0)
\end{equation}
with an appropriate time $\eta$ {\em a balance law} for the blow-up solution ${\bf y}(t)$ for (\ref{ODE-original}).
\end{dfn}
The next step is to derive the collection of systems for $\tilde {\bf Y}(\eta) \equiv {\bf Y}(\eta) - {\bf Y}_0(\eta)$ by means of {\em inhomogeneous linear systems}, but we only mention a key concept towards our aim here.
%%%

\begin{dfn}[Blow-up variation matrix]\rm
\label{dfn-blow-up-variation matrix}
Suppose that a periodic solution ${\bf Y}_0(\eta)$ of the balance law (\ref{eq-balance}) is given.
We call the $T_\eta$-periodic matrix-valued function
\begin{equation}
\label{blow-up-variation-matrix}
A(\eta) = - \frac{1}{k}\Lambda_\alpha + D_{\bf Y} f_{\alpha, k}( {\bf Y}_0(\eta))
\end{equation}
the {\em blow-up variation matrix} for the blow-up solution ${\bf y}(t)$.
\end{dfn}
Later we shall investigate {\em characteristic multipliers} of the linear differential system 
\begin{equation*}
\frac{d{\bf v}}{d\eta} = A(\eta){\bf v}
\end{equation*}
and associated (invariant) vector bundles with the correspondence to objects for (\ref{desing-para}).

\subsection{Correspondence of periodic solutions}
Unlike the correspondence of stationary objects \cite{asym2}, we have to fix an appropriate time-scale $\eta$ in (\ref{eq-balance}).
First we shall summarize the main result in this section; the correspondence among periodic solutions of two systems.
\begin{thm}[Correspondence of periodic solutions]
\label{thm-correspondence-C0}
Let ${\bf Y} = {\bf Y}_0(\eta)$ be a periodic solution to (\ref{eq-balance}) satisfying $p_\alpha({\bf Y}_0(\eta)) > 0$ for all $\eta$ with the time-scaling 
\begin{equation}
\label{time-equiv-1}
\frac{d\tau}{d\eta} = p_\alpha({\bf Y}_0(\eta))^k.
\end{equation}
Then 
\begin{equation*}
{\bf x}_0(\tau(\eta)) := p_\alpha({\bf Y}_0(\eta))^{-\Lambda_\alpha} {\bf Y}_0(\eta)
\end{equation*}
solves (\ref{desing-para}) with $\gamma_{0, \ast} \equiv \{{\bf x}_0(\tau(\eta)) \}_{\eta \in [0, T_\eta]} \in P_{\mathcal{E}}(g)$.
\par
%%%%%
Conversely, let $\gamma_\ast \in P_{\mathcal{E}}(g)$ with the period $T_\tau > 0$ and $\gamma = \{{\bf x}_\ast(\tau)\}_{\tau \in [0,T_\tau]}$, and suppose that
\begin{equation*}
\int_0^{T_\tau} G({\bf x}_\ast (\tau)) d\tau > 0.
\end{equation*}
Then the function $r(\tau) := X(\tau)^{1/k}$, where $X(\tau)$ is the unique $T_\tau$-periodic solution to
\begin{equation}
\label{X-linear}
\frac{dX}{d\tau} =  k G({\bf x}_\ast(\tau))X - 1,
\end{equation}
is well-defined and positive for all $\tau \in [0,T_\tau]$. 
Moreover, 
\begin{equation*}
{\bf Y}_\ast(\eta(\tau)) := r(\tau)^{\Lambda_\alpha} {\bf x}_\ast(\tau)
\end{equation*}
is a periodic solution to the balance law (\ref{eq-balance}) with the period $T_\eta > 0$, where $\eta (\tau)$ is the time-scale transformation determined by
\begin{equation}
\label{time-equiv-2}
\frac{d\eta}{d\tau} = r(\tau)^{-k}.
\end{equation}
\end{thm}
In the rest of the present section, we provide a proof how to construct the correspondence.

\subsubsection{From balance laws}
We begin with transforming the identity (\ref{eq-balance}) to (\ref{desing-para}).
Let $\eta$ be a time-scale, and ${\bf Y} = {\bf Y}_0(\eta)$ be a time-periodic solution to (\ref{eq-balance}) with the period $T_\eta > 0$.
Based on the correspondence manner introduced in \cite{asym2}, let
\begin{equation*}
{\bf x}_0(\eta) := p_\alpha({\bf Y}_0(\eta))^{-\Lambda_\alpha} {\bf Y}_0(\eta).
\end{equation*}
Then we have
\begin{align*}
\frac{d}{d\eta}x_{0;i}(\eta) &=  \frac{1}{p_\alpha({\bf Y}_0)^{\alpha_i}} \frac{dY_{0;i}(\eta)}{d\eta} - \frac{\alpha_i}{2c}Y_{0;i}(\eta) \left( \sum_{j\in I_\alpha}  Y_{0;j}(\eta)^{2\beta_i} \right)^{-\frac{\alpha_i}{2c} - 1} \cdot \sum_{j\in I_\alpha} 2\beta_j Y_{0;j}(\eta)^{2\beta_j-1} \frac{dY_{0;j}(\eta)}{d\eta} \\
%%%
	&= \frac{1}{p_\alpha({\bf Y}_0)^{\alpha_i}} \left[   \frac{dY_{0;i}(\eta)}{d\eta} - \frac{\alpha_i Y_{0;i}(\eta) }{c} \frac{1}{p_\alpha({\bf Y})^{2c}} \sum_{j\in I_\alpha} \beta_j Y_{0;j}(\eta)^{2\beta_j-1} \frac{dY_{0;j}(\eta)}{d\eta}  \right] \\
%%%
	&= \frac{1}{p_\alpha({\bf Y}_0)^{\alpha_i}} \left[  \frac{dY_{0;i}(\eta)}{d\eta} - \frac{\alpha_i p_\alpha({\bf Y}_0)^{\alpha_i} x_{0;i}(\eta) }{c} \frac{1}{p_\alpha({\bf Y}_0)^{2c}} \sum_{j\in I_\alpha} \beta_j p_\alpha({\bf Y}_0)^{\alpha_j(2\beta_j-1)} x_{0;j}(\eta)^{2\beta_j-1} \frac{dY_{0;j}(\eta)}{d\eta}  \right] \\
%%%
	&= \frac{1}{p_\alpha({\bf Y}_0)^{\alpha_i}} \left[ \frac{dY_{0;i}(\eta)}{d\eta} - \alpha_i p_\alpha({\bf Y}_0)^{\alpha_i} x_{0;i}(\eta)\sum_{j\in I_\alpha} \frac{x_{0;j}(\eta)^{2\beta_j-1}}{\alpha_j} p_\alpha({\bf Y}_0)^{-\alpha_j}  \frac{dY_{0;j}(\eta)}{d\eta}  \right],
\end{align*}
where the quantities in the sum $\sum_{j\in I_\alpha}$ is identical with those through the sum $\sum_{j=1}^n$ because we have imposed $\beta_j = 0$ whenever $j\not \in I_\alpha$, namely $\alpha_j=0$.
Note that ${\bf x}_0(\eta)$ is periodic with the period $T_\eta$.
On the other hand, we also have 
\begin{align*}
\frac{dY_{0;j}(\eta)}{d\eta} &=  -\frac{\alpha_j}{k}Y_{0;j} + f_{j;\alpha, k}({\bf Y}_0)\\
 	&= -\frac{\alpha_j}{k} p_\alpha ({\bf Y}_0)^{\alpha_j} x_{0;j} + f_{j;\alpha, k}(p_\alpha({\bf Y}_0)^{\Lambda_\alpha}{\bf x}_0) \\
 	&= -\frac{\alpha_j}{k} p_\alpha ({\bf Y}_0)^{\alpha_j} x_{0;j} + f_{j;\alpha, k}(p_\alpha({\bf Y}_0)^{\Lambda_\alpha}{\bf x}_0) \\
 	&= -\frac{\alpha_j}{k} p_\alpha ({\bf Y}_0)^{\alpha_j} x_{0;j} + p_\alpha ({\bf Y}_0)^{k+\alpha_j} f_{j;\alpha, k}({\bf x}_0),
\end{align*}
where we have used the quasi-homogeneity of $f_{\alpha, k}$.
Substituting this identity into the above equation, we have
\begin{align*}
\frac{d}{d\eta}x_{0;i}(\eta) 
%%%
	&= \left\{ -\frac{\alpha_i}{k} x_{0;i} + p_\alpha ({\bf Y}_0)^k f_{i;\alpha, k}({\bf x}_0) \right\} - \alpha_i  x_{0;i}(\eta)\sum_{j\in I_\alpha} \frac{x_{0;j}(\eta)^{2\beta_j-1}}{\alpha_j}\left\{ -\frac{\alpha_j}{k} x_{0;j} + p_\alpha({\bf Y}_0)^k f_{j;\alpha, k}({\bf x}_0) \right\} \\
%%%
	&= \left\{ -\frac{\alpha_i}{k} x_{0;i} + p_\alpha({\bf Y}_0)^k f_{i;\alpha, k}({\bf x}_0) \right\} \\
	&\quad - \alpha_i  x_{0;i}(\eta) \left[ \sum_{j\in I_\alpha} \frac{x_{0;j}(\eta)^{2\beta_j-1}}{\alpha_j}\left\{ -\frac{\alpha_j}{k} x_{0;j}  \right\} 
	+ \sum_{j\in I_\alpha} \frac{x_{0;j}(\eta)^{2\beta_j-1}}{\alpha_j}  p_\alpha({\bf Y}_0)^k f_{j;\alpha, k}({\bf x}_0) \right] \\
	&= p_\alpha({\bf Y}_0)^k \left\{  f_{i;\alpha, k}({\bf x}_0)  - \alpha_i  x_{0;i}(\eta) \sum_{j\in I_\alpha} \frac{x_{0;j}(\eta)^{2\beta_j-1}}{\alpha_j}   f_{j;\alpha, k}({\bf x}_0)\right\}\\
%%%%%%%%
	&=  p_\alpha({\bf Y}_0(\eta))^k \left\{  f_{i;\alpha, k}({\bf x}_0)  - G({\bf x}_0(\eta))\alpha_i  x_{0;i}(\eta) \right\}.
\end{align*}
Summarizing, 
\begin{equation*}
\frac{d{\bf x_0}}{d\eta}(\eta) = p_\alpha({\bf Y}_0(\eta))^k \left[ f_{\alpha, k}({\bf x}_0)  - G({\bf x}_0(\eta) )\Lambda_\alpha {\bf x}_0 \right]. 
\end{equation*}
Now we shall set $\eta$ so that $\tau$ and $\eta$ satisfy (\ref{time-equiv-1}), which yields
\begin{equation*}
\frac{d{\bf x}_0}{d\tau} = f_{\alpha, k}({\bf x}_0(\tau))  - G({\bf x}_0(\tau) )\Lambda_\alpha {\bf x}_0(\tau).
\end{equation*}
In particular, 
\begin{equation*}
\frac{dt}{d\eta} = p_\alpha({\bf Y}_0(\eta))^k \frac{1}{2c}\{1+(2c-1)p_\alpha({\bf x}_0(\tau(\eta)))^{2c}\} (1- p_\alpha({\bf x}_0(\tau(\eta)))^{2c})^k.
\end{equation*}

\subsubsection{From dynamics at infinity}
Next let $\gamma_\ast \in P_{\mathcal{E}}(g)$ with the period $T_\tau > 0$ and $\gamma = \{{\bf x}_\ast(\tau)\}_{\tau \in [0,T_\tau]}$.
In particular, we have
\begin{equation*}
\frac{d{\bf x}_\ast}{d\tau} = f_{\alpha, k}({\bf x}_\ast)  - G({\bf x}_\ast(\tau) )\Lambda_\alpha {\bf x}_\ast.
\end{equation*}
Based on the treatment of stationary blow-ups \cite{asym2}, {\rm formally} let
\begin{equation*}
{\bf Y}_\ast(\tau) := r(\tau)^{\Lambda_\alpha} {\bf x}_\ast(\tau).
\end{equation*}
Now the quantity $r(\tau)$ can depend on time $\tau$, and is determined later.
First, from the quasi-homogeneity, we have
\begin{equation*}
f_{\alpha, k}({\bf x}_\ast (\tau)) = r(\tau)^{-(kI + \Lambda_\alpha)}f_{\alpha, k}({\bf Y}_\ast (\tau(\eta)))
\end{equation*}
and hence
\begin{align*}
f_{\alpha, k}({\bf x}_\ast)  - G({\bf x}_\ast(\tau) )\Lambda_\alpha {\bf x}_\ast &= r(\tau)^{-(kI + \Lambda_\alpha)} f_{\alpha, k}({\bf Y}_\ast)  - G({\bf x}_\ast(\tau) )\Lambda_\alpha r^{-\Lambda}{\bf Y}_\ast,\\
%%%%
G({\bf x}_\ast) &= \sum_{j\in I_\alpha} \frac{x_{\ast; j}(\tau)^{2\beta_j-1}}{\alpha_j}   f_{j;\alpha, k}({\bf x}_\ast) \\
	&= \sum_{j\in I_\alpha} \frac{r^{-\alpha_j(2\beta_j-1)}Y_{\ast; j}(\tau)^{2\beta_j-1}}{\alpha_j} r^{-(k+\alpha_j)} f_{j;\alpha, k}({\bf Y}_\ast)\\
	&= r^{-2c-k}\sum_{j\in I_\alpha} \frac{Y_{\ast; j}(\tau)^{2\beta_j-1}}{\alpha_j} f_{j;\alpha, k}({\bf Y}_\ast).
\end{align*}
Using the commutativity of $\Lambda_\alpha$ and $r^{-\Lambda_\alpha}$, we have
\begin{align*}
f_{\alpha, k}({\bf x}_\ast)  - G({\bf x}_\ast(\tau) )\Lambda_\alpha {\bf x}_\ast &= r(\tau)^{-(kI+\Lambda_\alpha)} \left[ f_{\alpha, k}({\bf Y}_\ast) - r^{-2c}\left\{ \sum_{j\in I_\alpha} \frac{Y_{\ast; j}(\tau)^{2\beta_j-1}}{\alpha_j} f_{j;\alpha, k}({\bf Y}_\ast) \right\} \Lambda_\alpha  {\bf Y}_\ast \right].
\end{align*}
Here we {\em impose}
\begin{equation}
\label{constraint-r}
r(\tau)^k G({\bf x}_\ast(\tau)) \equiv \frac{1}{k} + H(\tau).
\end{equation}
On the other hand, we have
\begin{align*}
\frac{dY_{\ast; i}}{d\tau} &= \frac{d}{d\tau}(r(\tau)^{\alpha_i}x_{\ast,i}(\tau)) 
%%%
	= x_{\ast,i}(\tau)(\alpha_i r(\tau)^{\alpha_i-1}) \frac{dr}{d\tau}+ r(\tau)^{\alpha_i} \frac{d}{d\tau}(x_{\ast,i}(\tau))\\
%%%
	&= \alpha_i Y_{\ast,i}(\tau) \frac{1}{r} \frac{dr}{d\tau}+ r(\tau)^{\alpha_i} \frac{d}{d\tau}(x_{\ast,i}(\tau)).
\end{align*}
Summarizing, 
\begin{align}
\notag
\frac{d{\bf Y}_\ast}{d\tau} &= \Lambda_\alpha {\bf Y}_\ast \frac{1}{r} \frac{dr}{d\tau}+ r(\tau)^{\Lambda_\alpha} \frac{d}{d\tau}(x_{\ast,i}(\tau)) \\
%%%
\label{x_to_Y_pre}
 & = \Lambda_\alpha {\bf Y}_\ast\frac{1}{r} \frac{dr}{d\tau}+ r(\tau)^{-k} \left\{ f_{\alpha, k}({\bf Y}_\ast) - \left(\frac{1}{k} + H(\tau)\right) \Lambda_\alpha  {\bf Y}_\ast \right\}.
\end{align}
Our objective here is to {\em extract the balance law (\ref{eq-balance})}, which yields
\begin{equation*}
\frac{1}{r}\frac{dr}{d\tau} - r^{-k}H(\tau) = 0 \quad \Leftrightarrow \quad \frac{1}{k}\frac{d(r^k)}{d\tau} = H(\tau).
\end{equation*}
Because $H(\tau)$ is introduced satisfying (\ref{constraint-r}), we impose this constraint: 
\begin{equation*}
\frac{d(r^k)}{d\tau} =  kr^k G({\bf x}_\ast(\tau)) - 1.
\end{equation*}
Let $X(\tau) = r(\tau)^k$, which is described as a solution to the following inhomogeneous, linear nonautonomous differential equation (\ref{X-linear}).
%:
%\begin{equation}
%\label{X-linear}
%\frac{dX}{d\tau} =  k G({\bf x}_\ast(\tau))X - 1.
%\end{equation}
In particular, we have
\begin{equation*}
H(\tau) = G({\bf x}_\ast(\tau)) r(\tau)^k - \frac{1}{k}.
\end{equation*}
Note that $G({\bf x}_\ast(\tau))$ is $T_\tau$-periodic.
Therefore a general theory of linear inhomogeneous differential equations (e.g., \cite{C2006}) yields a characterization of a positive solution $X(\tau)$, as well as $r(\tau)$.

\begin{lem}
\label{lem-r}
Assume that, for (\ref{X-linear}),
\begin{equation}
\label{integral-exp}
\int_0^{T_\tau} G({\bf x}_\ast (\tau)) d\tau \not = 0. 
\end{equation}
Then (\ref{X-linear}) admits a unique $T_\tau$-periodic solution $X(\tau)$, which is bounded away from $0$.
In particular, for any $\tau$, quantities $\int_0^{T_\tau} G({\bf x}_\ast (\tau)) d\tau$ and $X(\tau)$ have identical signs.
\end{lem}
\begin{proof}
For simplicity, we shall write $T = T_\tau$.
Letting
\begin{equation*}
a_G(\tau) := k\int_0^{\tau} G({\bf x}_\ast(s))ds,\quad I(\tau) = \int_0^{\tau} e^{-a_G(s)}ds,
\end{equation*}
the general solution is written as
\begin{equation*}
X(\tau) = e^{a_G(\tau)} (C- I(\tau)).
\end{equation*}
Indeed, we have
\begin{align*}
\frac{dX}{d\tau}(\tau) &= \frac{da_G}{d\tau}(\tau) e^{a_G(\tau)} (C- I(\tau)) + e^{a_G(\tau)} \frac{dI}{d\tau}(\tau) 
%%%
	= kG({\bf x}_\ast(\tau)) X - 1.
\end{align*}
Using the $T_\tau$-periodicity of $X(\tau)$, we have
\begin{equation*}
e^{a_G(T)} (C- I(T)) = e^{a_G(0)} (C- I(0)) = C,
\end{equation*}
and hence
\begin{equation*}
C = e^{a_G(T)} \frac{I(T)}{e^{a_G(T)} - 1}.
\end{equation*}
By assumption, we have $a_G(T)\not = 0$, in particular $e^{a_G(T)} \not = 1$ holds, implying that the constant $C$ is well-defined.
Next assume that there is a time $\tau$ such that $X(\tau) = 0$, which means that
\begin{equation*}
C = e^{a_G(T)} \frac{I(T)}{e^{a_G(T)} - 1} = I(\tau).
\end{equation*}
Now $I(\tau)$ is strictly increasing because $ e^{-a_G(\tau)} > 0$, and we have
\begin{equation*}
0 < I(\tau) < I(T), \quad \text{ for all } \tau \in (0,T)
\end{equation*}
The sign of $X(\tau)$ is therefore determined as follows.
\begin{description}
\item[Case 1: $\int_0^T G({\bf x}_\ast (\tau)) d\tau > 0$.] 
\end{description}
Because $e^{a_G(T)} > 1$, we have
\begin{equation*}
C = e^{a_G(T)}\frac{I(T)}{e^{a_G(T)} - 1} > I(T) \geq I(\tau), \quad \text{ for all } \tau \in [0,T],
\end{equation*}
which indicates $X(\tau) > 0$.
\begin{description}
\item[Case 2: $\int_0^T G({\bf x}_\ast (\tau)) d\tau < 0$.] 
\end{description}
Because $e^{a_G(T)} < 1$, we have
\begin{equation*}
C = e^{a_G(T)}\frac{I(T)}{e^{a_G(T)} - 1} = -e^{a_G(T)}\frac{I(T)}{1 - e^{a_G(T)}} < 0 \leq I(\tau), \quad \text{ for all } \tau \in [0,T],
\end{equation*}
which indicates $X(\tau) < 0$ and the proof is completed.
\end{proof}
From this lemma, $X(\tau) > 0$ is always guaranteed provided that the quantity (\ref{integral-exp}) is positive, in which case $r(\tau) = X(\tau)^{1/k}$ is well-defined for any $\tau$.
Introducing (\ref{time-equiv-2}):
\begin{equation*}
\frac{d\eta}{d\tau} = r(\tau)^{-k},
\end{equation*}
our system (\ref{x_to_Y_pre}) is written as
\begin{align*}
\frac{d{\bf Y}_\ast}{d\eta} = f_{\alpha, k}({\bf Y}_\ast(\tau(\eta))) - \frac{1}{k} \Lambda_\alpha  {\bf Y}_\ast(\tau(\eta)),
\end{align*}
namely, the balance law (\ref{eq-balance}) with the time variable $\eta$.
Note that (\ref{time-equiv-1}) and (\ref{time-equiv-2}) yields the following identity:
\begin{equation*}
p_\alpha({\bf Y}_0(\eta)) = r(\tau(\eta)).
\end{equation*}
Summarizing the above arguments, the proof of Theorem \ref{thm-correspondence-C0} is completed.

\begin{rem}
The condition that (\ref{integral-exp}) is positive implies 
\begin{equation*}
\mu\equiv \exp \left( - \int_0^{T_\tau} G({\bf x}_\ast(\tau)) d\tau \right) < 1,
\end{equation*}
which corresponds to the observation that the local stable manifold of $\gamma_\ast$ is distributed in $\mathcal{D}$, which is already prescribed in the case of stationary blow-ups \cite{asym2}.
\end{rem}

\subsubsection{Remark on $G({\bf x})$}

One remark on $G({\bf x})$ related to a condition (\ref{integral-exp}) from the balance law is provided here.
Assume that ${\bf Y}_0(\eta)$ is a periodic solution to (\ref{eq-balance}) with period $T_\eta > 0$, and let
\begin{equation*}
{\bf x}_0(\tau(\eta)) := p_\alpha({\bf Y}_0(\eta))^{-\Lambda_\alpha} {\bf Y}_0(\eta),\quad \frac{d\tau}{d\eta} = p_\alpha({\bf Y}_0(\eta))^k
\end{equation*}
Now ${\bf x}_0(\tau(\eta))$ is $T_\tau$-periodic, where $T_\tau$ is determined by 
\begin{equation}
\label{period}
T_\tau = \int_0^{T_\eta} p_\alpha({\bf Y}_0(\eta))^k d\eta
\end{equation}
via (\ref{time-equiv-1}).
Moreover, let
\begin{equation*}
r(\tau(\eta)) := p_\alpha({\bf Y}_0(\eta)),\quad X(\tau) := r(\tau)^k
\end{equation*}
Then $X(\tau) = X(\tau(\eta)) > 0$ always holds provided $p_\alpha({\bf Y}_0(\eta))>0$ for all $\eta$.
Moreover, $X(\tau)$ is $T_\tau$-periodic satisfying
\begin{equation*}
\frac{dX}{d\tau} = kG({\bf x}_0(\tau))X - 1.
\end{equation*}
Dividing both sides by $X$, 
\begin{equation*}
\frac{d}{d\tau}\log X(\tau) =  kG({\bf x}_0(\tau)) - \frac{1}{X(\tau)}.
\end{equation*}
Here integrate both sides in $\tau$ from $0$ to $T_\tau$.
The integral of the left-hand side becomes $0$ from periodicity, and we therefore have
\begin{equation*}
k\int_0^{T_\tau} G({\bf x}_0(\tau)) d\tau = \int_0^{T_\tau} \frac{1}{X(\tau)}d\tau.
\end{equation*}
Because 
\begin{equation*}
\frac{1}{X(\tau)}d\tau = d\eta,
\end{equation*}
we finally have
\begin{equation}
\label{period-eta}
k\int_0^{T_\tau} G({\bf x}_0(\tau)) d\tau = \int_0^{T_\eta} d\eta = T_\eta > 0.
\end{equation}

As a consequence, we have the following criterion of (\ref{integral-exp}), in particular its positivity, from the viewpoint of the balance law.
\begin{cor}
\label{cor-from-balance-to-desing}
Let ${\bf Y}_0(\eta)$ be a periodic solution to the balance law (\ref{eq-balance}) with the period $T_\eta > 0$ satisfying $p_\alpha({\bf Y}_0(\eta))>0$ on the interval $[0, T_\eta]$.
Then the corresponding periodic orbit $\gamma_{0,\ast} = \{{\bf x}_0(\tau(\eta))\}_{\eta\in [0,T_\eta]} \in P_{\mathcal{E}}(g)$ satisfies (\ref{integral-exp}), in particular the quantity is positive.
As a consequence, the solution ${\bf Y}_\ast(\eta)$ constructed from ${\bf x}_0(\tau(\eta))$ is identical with ${\bf Y}_0(\eta)$.
\end{cor}

\section{Correspondence of invariant subbundles and characteristic multipliers}
\label{section-correspondence-2}

Before discussing the correspondence of linearized structure, a brief review of characteristic multipliers is provided.
Consider a differential operator 
\begin{equation*}
\mathcal{L}(\lambda){\bf w}(t) = (A(t) - \lambda I ){\bf w}(t) - \frac{d{\bf w}}{dt}(t)
\end{equation*}
on $C(\mathbb{R}, \mathbb{R}^n)$.
Let
\begin{equation*}
C_T(\mathbb{R}, \mathbb{R}^n) = \{{\bf w}\in C(\mathbb{R}, \mathbb{R}^n) \mid {\bf w}(t+T) = {\bf w}(t)\}
\end{equation*}
As a linear subspace of $C(\mathbb{R}, \mathbb{R}^n)$.
Let the kernel of $\mathcal{L}(\lambda)$ restricted to the linear subspace of differentiable functions
$C_T^1(\mathbb{R}, \mathbb{R}^n)$ be
\begin{equation*}
\ker_{\rm per}\mathcal{L}(\lambda) := \ker \mathcal{L}(\lambda)\cap C_T^1(\mathbb{R}, \mathbb{R}^n).
\end{equation*}
Then we have the following.
\begin{prop}[e.g., \cite{C2006}]
\label{prop-Floquet}
For some $T>0$, let $A(t)$ be a continuous, $T$-periodic matrix-valued function: $A(t+T) = A(t)$.
Then, for $\lambda \in \mathbb{C}$, the following are equivalent:
\begin{itemize}
\item $\ker_{\rm per}\mathcal{L}(\lambda) \not = \{0\}$.
\item $\mu = e^{\lambda T}$ is a characteristic multiplier of 
the linear differential equation
\begin{equation*}
{\bf v}' = A(t){\bf v},\quad {}' = \frac{d}{dt}.
\end{equation*}
\end{itemize}
\end{prop}

Now the correspondence of \lq\lq linearized" structures is investigated.
To this end, like the stationary case \cite{asym2}, we impose the following assumption.
\begin{ass}
\label{ass-diff-QH}
For the desingularized vector field $g$ on $\overline{\mathcal{D}}$, $D_{\bf x}\tilde f({\bf x}) = D_{\bf x}\tilde f_{\alpha, k}({\bf x})$ holds for any ${\bf x}\in \mathcal{E}$.
\end{ass}

\subsection{Transversal Invariant subbundles}

\subsubsection{From balance laws}
We begin with extracting an invariant structure for the balance law, based on {\em a common feature} of periodic solutions and {\em a structure associated with quasi-homogeneity of $f_{\alpha, k}$} observed in \cite{asym2}.
Let ${\bf Y}(\eta) = {\bf Y}_0(\eta)$ be a time-periodic solution to (\ref{eq-balance}) with the period $T_\eta > 0$, and $\gamma_0 := \{ {\bf Y}_0(\eta)\}_{\eta \in [0,T_\eta]}$. 
Here we consider the variational equation along ${\bf Y}_0(\eta)$:
\begin{equation}
\label{var-balance}
{\bf V}' = A(\eta) {\bf V},\quad A(\eta) =  -\frac{1}{k}\Lambda_\alpha + D_{\bf Y}f_{\alpha, k}({\bf Y}_0(\eta)).
\end{equation}
First, a function $e_1(\eta):= {\bf Y}_0' (\eta)$ is a solution to (\ref{var-balance}) by a general property of periodic solutions:
\begin{equation*}
e_1'(\eta) = A(\eta)e_1(\eta).
\end{equation*}
Next let $e_2(\eta) := \Lambda_\alpha {\bf Y}_0(\eta)$ as a candidate of invariant structure (cf. \cite{asym2}), and let
\begin{equation*}
R_{\bf Y}^{(2)} (\eta) := {\rm span}\{e_1(\eta), e_2(\eta)\}.
\end{equation*}
Then we have
\begin{align*}
A(\eta) e_2(\eta) &= -\frac{1}{k}\Lambda_\alpha (\Lambda_\alpha {\bf Y}_0) + D_{\bf Y}f_{\alpha, k}({\bf Y}_0(\eta))\Lambda_\alpha {\bf Y}_0\\
	&= -\frac{1}{k}\Lambda_\alpha (\Lambda_\alpha {\bf Y}_0) + (kI + \Lambda_\alpha) f_{\alpha, k}({\bf Y}_0(\eta)) \\
	&= \Lambda_\alpha \left( -\frac{1}{k}\Lambda_\alpha {\bf Y}_0(\eta) + f_{\alpha, k}({\bf Y}_0(\eta) \right) + kf_{\alpha, k}({\bf Y}_0(\eta))\\
%%%
	&= \Lambda_\alpha {\bf Y}_0'(\eta) + k \left\{ {\bf Y}_0'(\eta) + \frac{1}{k}\Lambda_\alpha {\bf Y}_0(\eta) \right\}\\
%%%
	&= e_2'(\eta) + k e_1'(\eta) + e_2(\eta).
\end{align*}
Letting
\begin{equation}
\label{2dim-balance}
E(\cdot) := [e_1(\cdot), e_2(\cdot)] : \mathbb{R}\to M_{n,2}(\mathbb{R}),
\end{equation}
the above identity implies
\begin{equation*}
A(\eta)E(\cdot) = E'(\eta) + E(\eta)\begin{pmatrix}
0 & k \\ 0 & 1
\end{pmatrix}
\end{equation*}
and we have
\begin{equation*}
A(\eta)E(\eta) - E'(\eta) = E(\eta) \begin{pmatrix}
0 & k \\ 0 & 1
\end{pmatrix} \in R_{\bf Y}^{(2)} (\eta).
\end{equation*}
This inclusion implies that $R_{\bf Y}^{(2)} (\eta)$ is an invariant fiber of $\mathbb{R}^n$.
As a consequence, we have the following proposition.

\begin{prop}
\label{prop-balance-inv-subbundle}
Let ${\bf Y}(\eta) = {\bf Y}_0(\eta)$ be a time-periodic solution to (\ref{eq-balance}).
Assume that the matrix $E(\eta)$ in (\ref{2dim-balance}) has a full rank at a time $\eta = \eta_0$.
Then the vector bundle
\begin{equation*}
R_{\bf Y}^{(2)} := \bigcup_{\eta\in [0,T_\eta]}{\rm span} \left\{\dot {\bf Y}_0(\eta), \Lambda_\alpha {\bf Y}_0(\eta)\right\}
\end{equation*}
is a $2$-dimensional invariant subbundle of $T_{\gamma_0}\mathbb{R}^n$ associated with (\ref{var-balance}).
\end{prop}

\begin{rem}[Linear independence of fibers]
From the general theory of linear differential systems (e.g. \cite{C2006}, Section 2.1.3), under the assumption in Proposition \ref{prop-balance-inv-subbundle}, the matrix $E(\eta)$ in (\ref{2dim-balance}) has a full rank for all $\eta$.
\end{rem}

Now we can easily calculate characteristic multipliers for $R_{\bf Y}^{(2)}$.
Assuming that $E(\eta)$ has a full rank, any vector $V\in R_{\bf Y}^{(2)}(\eta)$ is written by
\begin{equation*}
V(\eta) = E(\eta)z(\eta) = z_1(\eta)e_1(\eta) + z_2(\eta)e_2(\eta)
\end{equation*}
through $z(\eta) = (z_1(\eta), z_2(\eta))^T \in \mathbb{R}^2$.
Substituting this expression into the variational equation (\ref{var-balance}),
we get a $2$-dimensional ODE for $z$:
\begin{equation*}
z' = \begin{pmatrix}
0 & k \\ 0 & 1
\end{pmatrix}z,
\end{equation*}
and we have
\begin{equation*}
z_1(\eta) = C_1 + kC_2(e^\eta - 1),\quad z_2(\eta) = C_2e^{\eta}.
\end{equation*}

\begin{cor}
\label{cor-ch-balance}
Under the same setting as Proposition \ref{prop-balance-inv-subbundle}, the variational equation (\ref{var-balance}) admits characteristic multipliers $1$ and $e^{T_\eta} > 1$.
\end{cor}

\subsubsection{From dynamics at infinity}
Next we shall construct an invariant subbundle for the desingularized vector field $g$.
Let $\gamma_\ast =\{{\bf x}_\ast(\tau)\} \in P_{\mathcal{E}}(g)$ with the period $T_\tau > 0$.
In particular, ${\bf x}_\ast(\tau)$ satisfies
\begin{align}
\label{per-horizon}
\dot {\bf x}_\ast(\tau) &= \tilde f_{\alpha,k}({\bf x}_\ast(\tau)) - G({\bf x}_\ast(\tau))\Lambda_\alpha {\bf x}_\ast(\tau).
\end{align}
Here we consider the variational equation of (\ref{desing-para}) (not (\ref{per-horizon}) !) along with ${\bf x}_\ast(\tau)$:
\begin{equation}
\label{var-desing}
\dot {\bf V} = D_{\bf x}g({\bf x}_\ast(\tau)) {\bf V}.
\end{equation}
The Jacobian matrix of $g$ is
\begin{align*}
D_{\bf x}g({\bf x}) &= (2c-1)p_\alpha({\bf x})^{2c-1}\tilde f({\bf x})\nabla_{\bf x}p_\alpha({\bf x})^T \\
	&\quad + \left( 1 - \frac{2c-1}{2c}(1-p_\alpha({\bf x})^{2c})\right) D_{\bf x}\tilde f({\bf x})
	 - \Lambda_\alpha {\bf x} \nabla_{\bf x}G({\bf x})^T - G({\bf x})\Lambda_\alpha,
\end{align*}
while the matrix attained at ${\bf x}_\ast(\tau)$ is
\begin{align*}
D_{\bf x}g({\bf x}_\ast(\tau)) &= (2c-1)\tilde f({\bf x}_\ast(\tau))\nabla_{\bf x}p_\alpha({\bf x}_\ast(\tau))^T + D_{\bf x}\tilde f({\bf x}_\ast(\tau)) - \Lambda_\alpha {\bf x}_\ast(\tau) \nabla_{\bf x}G({\bf x}_\ast(\tau))^T - G({\bf x}_\ast(\tau))\Lambda_\alpha.
\end{align*}
Letting 
\begin{equation*}
C_\ast(\tau) := G({\bf x}_\ast(\tau)),
\end{equation*} 
$D_{\bf x}g({\bf x}_\ast(\tau))$ is written as
\begin{align}
\notag
\nabla_{\bf x}G({\bf x}_\ast(\tau)) &= \frac{d}{d\tau}\nabla_{\bf x}p_\alpha({\bf x}_\ast(\tau)) + 2cC_\ast(\tau)  \nabla_{\bf x}p_\alpha({\bf x}_\ast(\tau)) + A_g(\tau)^T \nabla_{\bf x}p_\alpha({\bf x}_\ast(\tau)),\\
%%%%
\label{Ag_tau}
A_g(\tau) &= -C_\ast(\tau) \Lambda_\alpha + D_{\bf x}\tilde f({\bf x}_\ast(\tau)).
\end{align}
Letting
\begin{equation*}
{\bf n}_\ast (\tau) := \nabla_{\bf x}p_\alpha({\bf x}_\ast(\tau)),
\end{equation*}
we have
\begin{align*}
D_{\bf x}g({\bf x}_\ast(\tau)) &= (2c-1)\tilde f({\bf x}_\ast(\tau))\nabla_{\bf x}p_\alpha({\bf x}_\ast(\tau))^T + D_{\bf x}\tilde f({\bf x}_\ast(\tau)) - \Lambda_\alpha {\bf x}_\ast(\tau ) \dot {\bf n}_\ast (\tau)^T \\
	&\quad - \Lambda_\alpha {\bf x}_\ast(\tau) \{ 2cC_\ast(\tau)  \nabla_{\bf x}p_\alpha({\bf x}_\ast(\tau)) + A_g(\tau)^T \nabla_{\bf x}p_\alpha({\bf x}_\ast(\tau)) \}^T - G({\bf x}_\ast(\tau))\Lambda_\alpha \\
%%%%%
	&= (2c-1)\tilde f({\bf x}_\ast(\tau))\nabla_{\bf x}p_\alpha({\bf x}_\ast(\tau))^T + D_{\bf x}\tilde f({\bf x}_\ast(\tau)) - \Lambda_\alpha {\bf x}_\ast(\tau ) \dot {\bf n}_\ast (\tau)^T  \\
	&\quad - \Lambda_\alpha {\bf x}_\ast(\tau) \{ 2cC_\ast(\tau) \nabla_{\bf x}p_\alpha({\bf x}_\ast(\tau))^T + \nabla_{\bf x}p_\alpha({\bf x}_\ast(\tau))^T A_g(\tau) \} - G({\bf x}_\ast(\tau))\Lambda_\alpha \\
%%%%%
	&= (2c-1)\tilde f_{\alpha, k}({\bf x}_\ast(\tau))\nabla_{\bf x}p_\alpha({\bf x}_\ast(\tau))^T + D_{\bf x}\tilde f_{\alpha, k}({\bf x}_\ast(\tau)) -  \{ 2cC_\ast(\tau) P_\ast(\tau) + P_\ast(\tau) A_g(\tau) \} - C_\ast(\tau)\Lambda_\alpha\\
%%%%%
	&\quad - \Lambda_\alpha {\bf x}_\ast(\tau )\dot {\bf n}_\ast (\tau)^T,
\end{align*}
where
\begin{equation*}
P_\ast(\tau) = \Lambda_\alpha {\bf x}_\ast(\tau) {\bf n}_\ast (\tau)^T
\end{equation*}
First note that $\varepsilon_1(\tau) := \dot {\bf x}_\ast(\tau)$ satisfies
\begin{equation*}
\dot \varepsilon_1(\tau) = D_{\bf x}g({\bf x}_\ast(\tau))\varepsilon_1(\tau)
\end{equation*}
because ${\bf x}_\ast(\tau)$ is a periodic solution.
Next note that
\begin{equation*}
(k+ \Lambda_\alpha)\tilde f_{\alpha, k}({\bf x}_\ast(\tau)) = D_{\bf x}\tilde f_{\alpha, k}({\bf x}_\ast(\tau))\Lambda_\alpha {\bf x}_\ast(\tau)
\end{equation*}
follows from the quasi-homogeneity of $\tilde f_{\alpha, k}$ and Assumption \ref{ass-diff-QH}.
Letting
\begin{equation*}
\varepsilon_2(\tau) := \Lambda_\alpha {\bf x}_\ast(\tau)
\end{equation*}
as an analogy from the balance law and \cite{asym2}, we calculate $D_{\bf x}g({\bf x}_\ast(\tau)) \varepsilon_2(\tau) - \dot \varepsilon_2(\tau)$: 
\begin{align*}
&D_{\bf x}g({\bf x}_\ast(\tau)) \varepsilon_2(\tau)\\
%%%%
&\quad = (2c-1)\tilde f_{\alpha, k}({\bf x}_\ast(\tau))\nabla_{\bf x}p_\alpha({\bf x}_\ast(\tau))^T \varepsilon_2(\tau) + D_{\bf x}\tilde f_{\alpha, k}({\bf x}_\ast(\tau))\varepsilon_2(\tau) \\
	&\quad \quad -  \{ 2cC_\ast(\tau) P_\ast(\tau)\varepsilon_2(\tau) + P_\ast(\tau) A_g(\tau)\varepsilon_2(\tau) \} - C_\ast(\tau)\Lambda_\alpha \varepsilon_2(\tau) - \varepsilon_2(\tau) \dot {\bf n}_\ast (\tau)^T \varepsilon_2(\tau) \\
%%%%
&\quad = (2c-1)\tilde f_{\alpha, k}({\bf x}_\ast(\tau))\nabla_{\bf x}p_\alpha({\bf x}_\ast(\tau))^T \varepsilon_2(\tau) + A_g(\tau) \varepsilon_2(\tau) -  2cC_\ast(\tau) P_\ast(\tau)\varepsilon_2(\tau) - P_\ast(\tau) A_g(\tau)\varepsilon_2(\tau) \\
	&\quad \quad - \varepsilon_2(\tau) \dot {\bf n}_\ast (\tau)^T \varepsilon_2(\tau).
\end{align*}
By using (\ref{per-horizon}), we have
\begin{align*}
&D_{\bf x}g({\bf x}_\ast(\tau)) \varepsilon_2(\tau)\\
%%%%
&\quad = (2c-1) \left\{ \varepsilon_1(\tau) + C_\ast(\tau)\varepsilon_2(\tau)\right\} \nabla_{\bf x}p_\alpha({\bf x}_\ast(\tau))^T \varepsilon_2(\tau) + (I-P_\ast(\tau)) A_g(\tau) \varepsilon_2(\tau) - 2cC_\ast(\tau) P_\ast(\tau)\varepsilon_2(\tau) \\
	&\quad \quad  - \varepsilon_2(\tau) \dot {\bf n}_\ast (\tau)^T \varepsilon_2(\tau) \\
%%%%
&\quad = (2c-1)\varepsilon_1(\tau)\nabla_{\bf x}p_\alpha({\bf x}_\ast(\tau))^T \varepsilon_2(\tau) + (2c-1) C_\ast(\tau)P_\ast(\tau) \varepsilon_2(\tau) + (I-P_\ast(\tau)) A_g(\tau) \varepsilon_2(\tau) -  2cC_\ast(\tau) P_\ast(\tau)\varepsilon_2(\tau) \\
	&\quad \quad - \varepsilon_2(\tau) \dot {\bf n}_\ast (\tau)^T \varepsilon_2(\tau) \\
%%%%
&\quad = (2c-1)\varepsilon_1(\tau)\nabla_{\bf x}p_\alpha({\bf x}_\ast(\tau))^T \varepsilon_2(\tau) + (I-P_\ast(\tau)) A_g(\tau) \varepsilon_2(\tau) -  C_\ast(\tau) P_\ast(\tau)\varepsilon_2(\tau) - \varepsilon_2(\tau) \dot {\bf n}_\ast (\tau)^T \varepsilon_2(\tau).
\end{align*}
Here note that ${\bf x}_\ast(\tau)\in \mathcal{E}$ implies
\begin{equation*}
{\bf n}_\ast (\tau)^T \varepsilon_2(\tau) \equiv 1,\quad P_\ast(\tau)\varepsilon_2(\tau) = \varepsilon_2(\tau)
\end{equation*}
for any $\tau$, equivalently,
\begin{align*}
D_{\bf x}g({\bf x}_\ast(\tau)) \varepsilon_2(\tau) = (2c-1)\varepsilon_1(\tau) + (I-P_\ast(\tau)) A_g(\tau) \varepsilon_2(\tau) -  C_\ast(\tau) \varepsilon_2(\tau) - \varepsilon_2(\tau)\dot {\bf n}_\ast (\tau)^T \varepsilon_2(\tau).
\end{align*}
Also using Assumption \ref{ass-diff-QH} and the quasi-homogeneity of $f_{\alpha, k}$, we have
\begin{align*}
A_g(\tau) \varepsilon_2(\tau) &= \left\{ -C_\ast(\tau) \Lambda_\alpha + D_{\bf x}\tilde f_{\alpha, k}({\bf x}_\ast(\tau)) \right\} \varepsilon_2(\tau) \\
%%%%
	&= -C_\ast(\tau) \Lambda_\alpha \varepsilon_2(\tau) + (k+\Lambda_\alpha) \tilde f_{\alpha, k}({\bf x}_\ast(\tau)) \\
	&= -C_\ast(\tau) \Lambda_\alpha \varepsilon_2(\tau) + (k+\Lambda_\alpha) \{ \varepsilon_1(\tau) + C_\ast(\tau) \varepsilon_2(\tau) \}  \\
	&= (k + \Lambda_\alpha)\varepsilon_1(\tau) + kC_\ast(\tau)\varepsilon_2(\tau).
\end{align*}
Then 
\begin{align*}
&D_{\bf x}g({\bf x}_\ast(\tau)) \varepsilon_2(\tau) \\
&\quad = (2c-1)\varepsilon_1(\tau) + (I-P_\ast(\tau)) \{ (k + \Lambda_\alpha)\varepsilon_1(\tau) + kC_\ast(\tau)\varepsilon_2(\tau) \}-  C_\ast(\tau) \varepsilon_2(\tau) - \varepsilon_2(\tau) \dot {\bf n}_\ast (\tau)^T \varepsilon_2(\tau) .
\end{align*}
Now $\varepsilon_1(\tau)\equiv \dot {\bf x}_\ast(\tau) \in T_{{\bf x}_\ast(\tau)}\mathcal{E}$ and, from observations in \cite{asym2}, the map $I-P_\ast(\tau)$ is the projection onto $T_{{\bf x}_\ast(\tau)}\mathcal{E}$.
Consequently, we have $(I-P_\ast(\tau))\varepsilon_1(\tau) = \varepsilon_1(\tau)$. 
On the other hand, we also have
\begin{equation*}
(I-P_\ast(\tau)) \varepsilon_2(\tau) = \varepsilon_2(\tau) -\varepsilon_2(\tau) {\bf n}_\ast (\tau)^T \varepsilon_2(\tau) = \varepsilon_2(\tau) - \varepsilon_2(\tau) = 0.
\end{equation*}
Therefore
\begin{align*}
D_{\bf x}g({\bf x}_\ast(\tau)) \varepsilon_2(\tau) &= (2c-1) \varepsilon_1(\tau) + k \varepsilon_1(\tau) + (I-P_\ast(\tau)) \dot \varepsilon_2(\tau) - C_\ast(\tau) \varepsilon_2(\tau) - \varepsilon_2(\tau) \dot {\bf n}_\ast (\tau)^T \varepsilon_2(\tau) 
\end{align*}
and hence
\begin{equation}
\label{Dge_dot_e}
D_{\bf x}g({\bf x}_\ast(\tau)) \varepsilon_2(\tau) - \dot \varepsilon_2(\tau) = (2c-1)\varepsilon_1(\tau) + k \varepsilon_1(\tau) -P_\ast(\tau) \dot \varepsilon_2(\tau) - C_\ast(\tau) \varepsilon_2(\tau) - \varepsilon_2(\tau)\dot {\bf n}_\ast (\tau)^T \varepsilon_2(\tau).
\end{equation}
Next we treat $P_\ast(\tau) \dot \varepsilon_2(\tau)$.
Because
\begin{equation*}
{\bf n}_\ast (\tau)^T \varepsilon_2(\tau)\equiv 1,\quad P_\ast(\tau) \varepsilon_2(\tau) = \varepsilon_2(\tau)
\end{equation*}
holds, differentiating both sides in the second identity in $\tau$, we have
\begin{equation*}
P_\ast(\tau) \dot \varepsilon_2(\tau) + \dot P_\ast(\tau) \varepsilon_2(\tau) = \dot \varepsilon_2(\tau).
\end{equation*}
Because
\begin{align*}
\dot P_\ast(\tau)\varepsilon_2(\tau) &= \frac{d}{d\tau} \left\{ \varepsilon_2(\tau) {\bf n}_\ast (\tau)^T \right\}\varepsilon_2(\tau)\\
	&= \dot \varepsilon_2(\tau) {\bf n}_\ast (\tau)^T \varepsilon_2(\tau) + \varepsilon_2(\tau)\dot {\bf n}_\ast (\tau)^T \varepsilon_2(\tau) \\
	&= \dot \varepsilon_2(\tau) + \varepsilon_2(\tau)\dot {\bf n}_\ast (\tau)^T \varepsilon_2(\tau),
\end{align*}
we have
\begin{align*}
P_\ast(\tau) \dot \varepsilon_2(\tau) &= - \dot P_\ast(\tau) \varepsilon_2(\tau) + \dot \varepsilon_2(\tau) \\
	&= - \dot \varepsilon_2(\tau) - \varepsilon_2(\tau) \dot {\bf n}_\ast(\tau)^T \varepsilon_2(\tau) + \dot \varepsilon_2(\tau)\\
	&= - \varepsilon_2(\tau)\dot {\bf n}_\ast(\tau)^T \varepsilon_2(\tau).
\end{align*}
Substituting this identity into (\ref{Dge_dot_e}), we have
\begin{equation*}
D_{\bf x}g({\bf x}_\ast(\tau)) \varepsilon_2(\tau) - \dot \varepsilon_2(\tau) = (2c-1)\varepsilon_1(\tau) + k \varepsilon_1(\tau) - C_\ast(\tau) \varepsilon_2(\tau).
\end{equation*}
According to an expression of invariant subbundles, we have obtained the following proposition.

\begin{prop}
\label{prop-desing-inv-subbundle}
Let $\gamma_\ast = \{{\bf x}_\ast(\tau)\} \in P_{\mathcal{E}}(g)$ for the desingularized vector field $g$ in (\ref{desing-para}).
Further Assumption \ref{ass-diff-QH} is imposed.
Then
\begin{equation*}
R^{(2)}_{\bf x} := \bigcup_{\tau\in [0, T_\tau]}{\rm span} \left\{\varepsilon_1(\tau), \varepsilon_2(\tau)\right\},\quad \varepsilon_1(\tau) := \dot {\bf x}_\ast(\tau), \quad \varepsilon_2(\tau):= \Lambda_\alpha {\bf x}_\ast(\tau),
\end{equation*}
is a $2$-dimensional invariant subbundle of $T_{\gamma_\ast}\mathbb{R}^n$ associated with (\ref{var-desing}).
\end{prop}

\begin{rem}[Linear independence of fibers]
Thanks to the invariant structure of the horizon $\mathcal{E}$, the linearly independent nature of vectors in the fiber of $R^{(2)}_{\bf x}$ is always clear. 
Indeed, because $\varepsilon_1(\tau) \equiv \dot {\bf x}_\ast(\tau)$ is a tangent vector of $\gamma_\ast$, it is an element of $T_{{\bf x}_\ast(\tau)}\mathcal{E}$.
On the other hand, the identity
\begin{equation*}
\nabla_{\bf x}p_\alpha({\bf x}_\ast(\tau))^T \varepsilon_2(\tau) = \nabla_{\bf x}p_\alpha({\bf x}_\ast(\tau))^T \Lambda_\alpha {\bf x}_\ast(\tau)\equiv 1
\end{equation*}
implies that two vectors $\varepsilon_1(\tau)$ and $\varepsilon_2(\tau)$ generating the fiber $R^{(2)}_{\bf x}(\tau)$ are always linearly independent.
However, there is no such structure in $R^{(2)}_{\bf Y}$ in general.
Therefore, the linear independence of vectors $e_1(\eta)$ and $e_2(\eta)$ is, even if it is required at a point $\eta_0$, just an {\em assumption} in the present framework.
\end{rem}

Next we shall consider the associated characteristic multipliers.
Let
\begin{equation*}
E_\ast(\tau) := [\varepsilon_1(\tau), \varepsilon_2(\tau)].
\end{equation*}
Then we have
\begin{equation*}
D_{\bf x}g({\bf x}_\ast(\tau))E_\ast(\tau) - E_\ast'(\tau) = E_\ast(\tau)\begin{pmatrix}
0 & (2c-1) + k \\ 0 & -C_\ast(\tau)
\end{pmatrix}.
\end{equation*}
Now any vector $W(\tau)\in R_{\bf x}^{(2)}(\tau)$ is expressed by 
\begin{equation*}
W(\tau) = E_\ast(\tau) z(\tau) = z_1(\tau)\varepsilon_1(\tau) + z_2(\tau)\varepsilon_2(\tau)
\end{equation*}
through a two-dimensional vector $z(\tau) = (z_1(\tau), z_2(\tau))^T$ for each $\tau$.
Substituting this expression into (\ref{var-desing}), we have
\begin{align*}
\frac{dW}{d\tau}(\tau) &= z_1'(\tau)\varepsilon_1(\tau) +  z_1(\tau) \varepsilon_1'(\tau) + z_2'(\tau)\varepsilon_2(\tau) + z_2(\tau)\varepsilon_2' (\tau),\\
\dot \varepsilon_1(\tau) &= D_{\bf x}g({\bf x}_\ast(\tau)) \varepsilon_1(\tau),\\
\dot \varepsilon_2(\tau) &= D_{\bf x}g({\bf x}_\ast(\tau)) \varepsilon_2(\tau) - (2c-1+k)\varepsilon_1(\tau) + C_\ast(\tau) \varepsilon_2(\tau),\\
D_{\bf x}g({\bf x}_\ast(\tau)) w_1(\tau) &= z_1(\tau)D_{\bf x}g({\bf x}_\ast(\tau))\varepsilon_1(\tau),\\
D_{\bf x}g({\bf x}_\ast(\tau)) w_2(\tau) &= z_2(\tau)D_{\bf x}g({\bf x}_\ast(\tau))\varepsilon_2(\tau),
\end{align*}
which yields
\begin{align*}
0 &= \frac{dW}{d\tau}(\tau) - D_{\bf x}g({\bf x}_\ast(\tau)) W(\tau)\\
	&= z_1'(\tau)\varepsilon_1(\tau) + z_2'(\tau)\varepsilon_2(\tau) - \{(2c-1+k)\varepsilon_1(\tau) + C_\ast(\tau) \varepsilon_2(\tau)\}z_2(\tau).
\end{align*}

By an argument similar to that used to obtain Corollary \ref{cor-ch-balance}, we have the following corollary, which shows an explicit calculation of several characteristic multipliers.
\begin{cor}
\label{cor-ch-desing}
Under the same setting as Proposition \ref{prop-desing-inv-subbundle}, the variational equation (\ref{var-desing}) admits characteristic multipliers
\begin{equation}
\label{ch-trans-desing}
1,\quad \exp\left(- \int_0^{T_\tau} G({\bf x}_\ast(\tau)) d\tau\right).
\end{equation}
\end{cor}

On the other hand, letting ${\bf x}_0 = r^{-\Lambda_\alpha} {\bf Y}_0$ from Section \ref{section-correspondence-1}, we have
\begin{align*}
G({\bf x}_0) &= \sum_{j\in I_\alpha} \frac{x_{0,j}^{2\beta_j - 1}}{\alpha_j} \tilde f_{j; \alpha, k}({\bf x}_0) \\
	&= \sum_{j\in I_\alpha} \frac{(r^{-\alpha_j} Y_{0,j})^{2\beta_j - 1}}{\alpha_j} \tilde f_{j; \alpha, k}(r^{-\Lambda_\alpha} {\bf Y}_0 ) \\
	&= \sum_{j\in I_\alpha} \frac{ (r^{-\alpha_j(2\beta_j - 1)} Y_{0,j}^{2\beta_j - 1}}{\alpha_j} r^{-(k+\alpha_j)} \tilde f_{j; \alpha, k}( {\bf Y}_0 ) \\
	&= r^{-(2c+k)} \sum_{j\in I_\alpha} \frac{  Y_{0,j}^{2\beta_j - 1}}{\alpha_j} \tilde f_{j; \alpha, k}( {\bf Y}_0 ) \\
	&= r^{-(2c+k)} G({\bf Y}_0).
\end{align*}
which yields
\begin{align*}
\int_0^{T_\tau} G({\bf x}_0 (\tau(\eta))) d\tau &= \int_0^{T_\eta} G({\bf x}_0 (\tau(\eta))) r(\eta)^k d\eta \\
	&= \int_0^{T_\eta} r^{-2c} G({\bf Y}_0 (\eta)) d\eta.
\end{align*}
From Corollary \ref{cor-from-balance-to-desing}, the left-hand side becomes positive, and hence the second multiplier in (\ref{ch-trans-desing}) is, when it is constructed from ${\bf Y}_0(\eta)$, always less than $1$.

\subsection{Tangent quotient subbundles}

Next we shall discuss the remaining correspondence of the \lq\lq eigenstructures".
In \cite{asym2} (as well as \cite{M2026_asym_nonaut}) we have introduced an auxiliary matrix $A_g$ connecting the formal and eigenstructures among $D_{\bf x}g({\bf x}_\ast)$ and $A$ associated with stationary solutions (in particular, objects to be considered are independent of $\tau$).
Here we follow this idea, namely connect \lq\lq eigenstructures" between $D_{\bf x}g({\bf x}_\ast(\tau))$ and $A(\eta)$ via $A_g(\tau)$ defined in (\ref{Ag_tau}).

\subsubsection{Operator identity}
First let
\begin{equation*}
P(\tau) \equiv P_\ast(\tau) := \varepsilon_2(\tau){\bf n}_\ast(\tau)^T.
\end{equation*}
Note that
\begin{equation*}
T_{{\bf x}_\ast(\tau)}\mathcal{E} = \ker ({\bf n}_\ast(\tau)^T)
\end{equation*}
holds for each $\tau$ and hence $I-P_\ast(\tau)$ is a (generally non-orthogonal) projection onto the tangent space $T_{{\bf x}_\ast(\tau)}\mathcal{E}$.
In particular, ${\bf n}_\ast(\tau)^T \xi(\tau) = 0$ holds whenever $\xi(\tau)\in T_{{\bf x}_\ast(\tau)}\mathcal{E}$.
Differentiating both sides of the above identity with respect to $\tau$, we have
\begin{equation*}
\dot {\bf n}_\ast(\tau)^T \xi(\tau) + {\bf n}_\ast(\tau)^T \dot \xi(\tau) = 0
\end{equation*}
Let 
\begin{equation*}
H(\tau) := [h_1(\tau), \ldots, h_{n-1}(\tau)]
\end{equation*}
be the matrix representation of base vectors of $T_{{\bf x}_\ast(\tau)}\mathcal{E}$.
Then we have
\begin{align*}
D_{\bf x}g({\bf x}_\ast(\tau))H(\tau)
%%%%%
	&= (2c-1)\tilde f_{\alpha, k}({\bf x}_\ast(\tau)){\bf n}_\ast(\tau)^T H(\tau) \\
	&\quad + D_{\bf x}\tilde f_{\alpha, k}({\bf x}_\ast(\tau))H(\tau) -  \{ 2cC_\ast(\tau) P_\ast(\tau)H(\tau) + P_\ast(\tau) A_g(\tau)H(\tau) \} \\
%%%%%
	&\quad - C_\ast(\tau)\Lambda_\alpha H(\tau) - \varepsilon_2(\tau) \dot {\bf n}_\ast(\tau)^T H(\tau)\\
%%%%%
	&= (I - P_\ast(\tau)) A_g(\tau)H(\tau) - \varepsilon_2(\tau) \dot {\bf n}_\ast(\tau)^T  H(\tau).
\end{align*}
On the other hand, orthogonality ${\bf n}_\ast(\tau)^T H(\tau) = 0$ also implies 
\begin{equation}
\dot {\bf n}_\ast(\tau)^T H(\tau) + {\bf n}_\ast(\tau)^T \dot H(\tau) = 0
\end{equation}
and hence their combination yields
\begin{equation*}
P\dot H(\tau) = \varepsilon_2 (\tau) {\bf n}_\ast(\tau)^T \dot H(\tau) = - \varepsilon_2 (\tau) \dot {\bf n}_\ast(\tau)^T H(\tau).
\end{equation*}
Therefore we obtain
\begin{align*}
D_{\bf x}g({\bf x}_\ast(\tau))H(\tau)
%%%%%
	&= (I - P_\ast(\tau)) A_g(\tau)H(\tau) + P(\tau) \dot H (\tau) 	
\end{align*}
and
\begin{align*}
D_{\bf x}g({\bf x}_\ast(\tau))H(\tau) - \dot H(\tau)
%%%%%
	&= (I - P_\ast(\tau)) A_g(\tau)H(\tau) - (I- P(\tau)) \dot H (\tau) \\
	&= (I - P_\ast(\tau))(A_g(\tau)H(\tau) - \dot H(\tau))
\end{align*}
Simple addition of the diagonal matrix $\lambda I$ for any $\lambda\in \mathbb{C}$ yields the following proposition, which is one of the main tools towards our correspondence.

\begin{prop}
\label{prop-Dg-I-P}
Under Assumption \ref{ass-diff-QH}, the following identity holds for any $\lambda \in \mathbb{C}$ and ${\bf w}: \mathbb{R}\to T_{{\bf x}_\ast(\cdot)}\mathcal{E}$:
\begin{equation}
\label{identity-Dg-Ag}
\{D_{\bf x}g({\bf x}_\ast(\tau)) -\lambda I \}{\bf w}(\tau) - \dot {\bf w}(\tau) = (I - P_\ast(\tau))\{ (A_g(\tau) - \lambda I ){\bf w}(\tau) - \dot {\bf w}(\tau)\}.
\end{equation}
\end{prop}

Note that the mapping ${\bf w}: \mathbb{R}\to T_{{\bf x}_\ast(\cdot)}\mathcal{E}$ mentioned in the above proposition is just a section on the tangent bundle $T\mathcal{E}$, namely ${\bf w}\in \Gamma(T\mathcal{E})$.
The identity (\ref{identity-Dg-Ag}) is then written as an identity of differential operators on $\Gamma(T\mathcal{E})$.
Letting
\begin{align*}
\mathcal{L}_g(\lambda) {\bf w} &:= \{D_{\bf x}g({\bf x}_\ast(\cdot)) -\lambda I \}{\bf w} - \frac{d{\bf w}}{d\tau},\\
\mathcal{L}_{A,g}(\lambda) {\bf w} &:= (A_g(\cdot) - \lambda I ){\bf w} - \frac{d{\bf w}}{d\tau},\\
\mathcal{L}_{A,g,\mathcal{E}}(\lambda) {\bf w} &:= (I - P_\ast(\cdot)) \mathcal{L}_{A,g}(\lambda) {\bf w},
\end{align*}
(\ref{identity-Dg-Ag}) implies the identity $\mathcal{L}_g(\lambda) = \mathcal{L}_{A,g,\mathcal{E}}(\lambda)$ on $\Gamma(T\mathcal{E})$.
Summarizing the above arguments, we obtain the following result.

\begin{cor}
Under Assumption \ref{ass-diff-QH}, we have the identity of differential operators on $\Gamma(T\mathcal{E})$:
\begin{equation*}
\mathcal{L}_g(\lambda)^N = \mathcal{L}_{A,g,\mathcal{E}}(\lambda)^N
\end{equation*}
for any $N\in \mathbb{Z}_{\geq 0}$ and $\lambda \in \mathbb{C}$.
\end{cor}

\subsubsection{Quotient bundles}

Next, we shall consider several subbundles of the tangent bundle $T\mathcal{E}$.
A basic property of periodic orbits indicates that the tangent vector $\varepsilon_1(\tau)\equiv \dot {\bf x}_\ast (\tau)$ generates the $1$-dimensional invariant subbundle $T\gamma_\ast$ of $T\mathcal{E}$ associated with (\ref{var-desing}).
On the other hand, the subbundle $T\gamma_\ast$ is {\em also} an invariant subbundle of $R^{(2)}_{\bf x}$ constructed in Proposition \ref{prop-desing-inv-subbundle}.
From the transversality of the remaining (linearly independent) direction of $R^{(2)}_{\bf x}$ with $\mathcal{E}$, we have
\begin{equation*}
T\mathcal{E} \cap R^{(2)}_{\bf x} = T\gamma_\ast.
\end{equation*}
%%%%%%
Then the {\em quotient bundle}
\begin{equation*}
\mathcal{Q}^{(2)}_{\bf x} \equiv T\mathcal{E} / T\gamma_\ast
\end{equation*}
can be defined, which will be the basis of the treatment of \lq\lq remaining linearized structures".
\par
Remember that the characteristic multipliers of $\gamma_\ast$ have the form
\begin{equation*}
\mu = e^{\lambda T_\tau},
\end{equation*}
while all solutions to the variational equation are described by the linear combination of functions of the form
\begin{equation*}
{\bf w}(\tau) = e^{\lambda \tau} \tilde {\bf w}(\tau),\quad \tilde {\bf w}(\tau + T_\tau) = \tilde {\bf w}(\tau).
\end{equation*}
Here the quantity $\lambda \in \mathbb{C}$ is known as a {\em Floquet exponent}.
Through the vector space of {\em periodic sections}
\begin{equation*}
\Gamma_{\rm per}(T\mathcal{E}) = \{{\bf w}\in \Gamma (T\mathcal{E}) \mid {\bf w}(\tau + T_\tau) = {\bf w}(\tau)\},
\end{equation*}
solution structure of the variational equation (\ref{var-desing}) can be completely described.

\subsection{Scaling of operators}

Before studying the nature of quotient bundles mentioned above,
consider the following differential operator associated with $A(\eta)$ as a counterpart of $\mathcal{L}_{A,g}(\lambda)$;
\begin{equation*}
\mathcal{L}_A(\lambda) {\bf v} := (A(\cdot) - \lambda I ){\bf v} - \frac{d{\bf v}}{d\eta}.
\end{equation*}
Correspondence of time variables $\tau$ and $\eta$ is provided by (\ref{time-equiv-1}) or (\ref{time-equiv-2}).
Here compare $A(\eta)$ with $A_g(\tau(\eta))$
Because $r(\tau(\eta))^{-\Lambda_\alpha}{\bf Y}_0(\eta) = {\bf x}_\ast(\tau(\eta))$ in the $\tau$-time-scale, we have
\begin{equation*}
r(\tau)^{-(kI+\Lambda_\alpha)} D_{\bf Y}f_{\alpha, k}({\bf Y}_0(\eta(\tau))) r(\tau)^{\Lambda_\alpha} = D_{\bf x}\tilde f_{\alpha, k}({\bf x}_\ast(\tau)) 
\end{equation*}
based on arguments in the stationary case \cite{asym2}.
On the other hand, we have obtained
\begin{equation*}
C_\ast(\tau) \equiv G({\bf x}_\ast(\tau)) = \frac{1}{kr(\tau)^k}\left\{ \frac{d(r^k)}{d\tau} + 1\right\}
\end{equation*}
in (\ref{X-linear}).
Using the commutativity of diagonal matrices, we have
\begin{align*}
A_g(\tau) &= -C_\ast(\tau)\Lambda_\alpha + D_{\bf x}\tilde f_{\alpha, k}({\bf x}_\ast(\tau))\\
	&= - \frac{1}{kr(\tau)^k} \left\{ \frac{d(r^k)}{d\tau} + 1 \right\}\Lambda_\alpha + r(\tau)^{-(kI+\Lambda_\alpha)} D_{\bf Y}f_{\alpha, k}({\bf Y}_0(\eta(\tau))) r(\tau)^{\Lambda_\alpha} \\
%%%%
	&= r(\tau)^{-k} \left[  - \frac{1}{k} \left\{ \frac{d(r^k)}{d\tau} + 1 \right\}\Lambda_\alpha + r(\tau)^{-\Lambda_\alpha} D_{\bf Y}f_{\alpha, k}({\bf Y}_0(\eta(\tau))) r(\tau)^{\Lambda_\alpha}\right] \\
%%%%
	&= r(\tau)^{-(kI+\Lambda_\alpha)} \left[  -r^{k-1} \frac{dr}{d\tau} \Lambda_\alpha  - \frac{1}{k} \Lambda_\alpha + D_{\bf Y}f_{\alpha, k}({\bf Y}_0(\eta(\tau))) \right] r(\tau)^{\Lambda_\alpha} \\
%%%%
	&= r(\tau)^{-(kI+\Lambda_\alpha)} A(\eta(\tau)) r(\tau)^{\Lambda_\alpha} -
	r(\tau)^{-1} \frac{dr}{d\tau} \Lambda_\alpha.
\end{align*}
As a result, for any ${\bf w} = {\bf w}(\tau)$, we have 
\begin{align*}
\mathcal{L}_{A,g}(\lambda){\bf w}(\tau) &\equiv (A_g(\tau) - \lambda I ){\bf w}(\tau) - \frac{d{\bf w}}{d\tau}(\tau)\\
%%%
	&=  A_g(\tau){\bf w}(\tau) - \lambda {\bf w}(\tau) - \frac{d{\bf w}}{d\eta}(\eta)\frac{d\eta}{d\tau}\\
%%%
	&=  r(\tau)^{-(kI+\Lambda_\alpha)} A(\eta(\tau)) r(\tau)^{\Lambda_\alpha}{\bf w}(\eta(\tau)) -
	r(\tau)^{-1} \frac{dr}{d\tau} \Lambda_\alpha {\bf w}(\eta(\tau)) - \lambda {\bf w}(\eta(\tau)) - \frac{d{\bf w}}{d\eta}(\eta(\tau))\frac{d\eta}{d\tau}.
\end{align*}
Moreover, using (\ref{time-equiv-2}),
\begin{align*}
\mathcal{L}_{A,g}(\lambda){\bf w}(\tau) 
	&=  r(\tau)^{-(kI+\Lambda_\alpha)} A(\eta(\tau)) r(\tau)^{\Lambda_\alpha}{\bf w}(\eta(\tau)) -
	r(\tau)^{-1} \frac{dr}{d\tau} \Lambda_\alpha {\bf w}(\eta(\tau)) - \lambda {\bf w}(\eta(\tau)) - r(\tau)^{-k}\frac{d{\bf w}}{d\eta}(\eta).
\end{align*}
Now let
\begin{equation*}
{\bf W}(\tau) := r(\tau)^{\Lambda_\alpha}{\bf w}(\tau).
\end{equation*}
Then
\begin{align*}
\frac{d{\bf w}}{d\tau} &= \frac{d}{d\tau} \left\{ r(\tau)^{-\Lambda_\alpha}{\bf W}(\tau)\right\}\\
	&= -r(\tau)^{-1} \Lambda_\alpha r(\tau)^{-\Lambda_\alpha} \frac{dr}{d\tau} {\bf W}(\tau) + r(\tau)^{-\Lambda_\alpha} \frac{d{\bf W}}{d\tau}(\tau),\\
\frac{d{\bf w}}{d\eta} &= \frac{d{\bf w}}{d\tau}\frac{d\tau}{d\eta}\\
	&= r(\tau)^k \frac{d{\bf w}}{d\tau}\\
%%%%%
	&=  r(\tau)^k\left\{ -r(\tau)^{-1} \Lambda_\alpha r(\tau)^{-\Lambda_\alpha} \frac{dr}{d\tau} {\bf W}(\tau) + r(\tau)^{-\Lambda_\alpha} \frac{d{\bf W}}{d\tau}(\tau) \right\}\\
	&= r(\tau)^k r(\tau)^{-\Lambda_\alpha} \left\{ -r(\tau)^{-1} \frac{dr}{d\tau} \Lambda_\alpha {\bf W}(\tau) + \frac{d{\bf W}}{d\eta}\frac{d\eta}{d\tau} \right\}\\
%%%%%
	&= r(\tau)^{-\Lambda_\alpha} \left\{ -r(\tau)^{-1} r(\tau)^k \frac{dr}{d\tau} \Lambda_\alpha {\bf W}(\tau) + \frac{d{\bf W}}{d\eta} \right\}.
\end{align*}
Substituting them into $\mathcal{L}_{A,g}(\lambda){\bf w}$, we have
\begin{align*}
\mathcal{L}_{A,g}(\lambda){\bf w}(\tau) 
%%%%%
	&=  r(\tau)^{-(kI+\Lambda_\alpha)} A(\eta(\tau)) r(\tau)^{\Lambda_\alpha}{\bf w}(\eta(\tau)) -
	r(\tau)^{-1} \frac{dr}{d\tau} \Lambda_\alpha {\bf w}(\eta(\tau)) - \lambda {\bf w}(\eta(\tau)) \\
	&\quad - r(\tau)^{-k} \left[ r(\tau)^{-\Lambda_\alpha} \left\{ -r(\tau)^{-1} r(\tau)^k \frac{dr}{d\tau} \Lambda_\alpha {\bf W}(\tau) + \frac{d{\bf W}}{d\eta} \right\} \right] \\
%%%%%
	&=  r(\tau)^{-(kI+\Lambda_\alpha)} A(\eta(\tau)) r(\tau)^{\Lambda_\alpha}{\bf w}(\eta(\tau)) -
	r(\tau)^{-1} \frac{dr}{d\tau} \Lambda_\alpha {\bf w}(\eta(\tau)) - \lambda {\bf w}(\eta(\tau)) \\
	&\quad + r(\tau)^{-\Lambda_\alpha} \left\{ r(\tau)^{-1} \frac{dr}{d\tau} \Lambda_\alpha {\bf W}(\tau) - r(\tau)^{-k} \frac{d{\bf W}}{d\eta} \right\}\\
%%%%%
	&=  r(\tau)^{-(kI+\Lambda_\alpha)} A(\eta(\tau)) r(\tau)^{\Lambda_\alpha}{\bf w}(\eta(\tau)) - \lambda {\bf w}(\eta(\tau)) - r(\tau)^{-(kI+\Lambda_\alpha)} \frac{d{\bf W}}{d\eta} \\
%%%%%
	&=  r(\tau)^{-(kI+\Lambda_\alpha)} \left\{ A(\eta(\tau)) {\bf W}(\eta) - \tilde \lambda {\bf W}(\eta) - \frac{d{\bf W}}{d\eta}\right\},
\end{align*}
where
\begin{equation*}
\tilde \lambda := r(\eta(\tau))^k \lambda.
\end{equation*}
Next we shall describe the above identity as an {\em operator equation}.
Let
\begin{equation*}
\mathcal{L}_{A}^{(2)}(\lambda){\bf W}(\eta) := (A(\eta) - r(\tau(\eta))^k \lambda I ){\bf W}(\eta) - \frac{d{\bf W}}{d\eta}(\eta).
\end{equation*}
%とする. $\mathcal{L}_{A}(\lambda)$と多少紛らわしいが(どこかで記号の変更を試みる), こちらでは係数$\lambda$に$r(\tau(\eta))^k$をかけていることを陽に表している. 
%構成としては$\mathcal{L}_{A}(\tilde \lambda)$そのものであるが, 念のための処置ということで. 
%さて, これを使うと作用素の関係式として
Then
\begin{equation}
\label{similar-A-Ag}
\mathcal{L}_{A,g}(\lambda) = r(\tau(\cdot))^{-(kI+\Lambda_\alpha)}\mathcal{L}_{A}^{(2)}(\lambda) r(\tau(\cdot))^{\Lambda_\alpha},
\end{equation}
each of which has a period $T_\tau$ and $T_\eta$, respectively.
%を得る. 各々, 周期$T_\tau$, $T_\eta$を持つ作用素である. 
%$T_\tau$, $T_\eta$は時間スケール変換(\ref{time-equiv-1})と(\ref{time-equiv-2})により決まるもので, 一致するとは限らない. 

\subsection{Characteristic multipliers}

Here the correspondence of characteristic multipliers is considered through the identity (\ref{similar-A-Ag}) as well as Proposition \ref{prop-Floquet}.
Suppose that ${\bf w}\in \ker_{\rm per}\mathcal{L}_{A,g}(\lambda)$ for some $\lambda\in \mathbb{C}$, in which case ${\bf w}(\tau + T_\tau) = {\bf w}(\tau)$.
Then we define
\begin{equation}
\label{w-v}
{\bf v}(\tau) := e^{\lambda \tau} {\bf w}(\tau).
\end{equation}
Direct calculations yield that ${\bf v} = {\bf v}(\tau)$ satisfies 
\begin{equation*}
\frac{d{\bf v}}{d\tau} = A_g(\tau){\bf v},
\end{equation*}
consequently ${\bf v}(\tau + T_\tau) = e^{\lambda T_\tau} {\bf v}(\tau)$. 
In other words, one characteristic multiplier is determined as
\begin{equation*}
\mu = e^{\lambda T_\tau}.
\end{equation*}
Next, based on ${\bf w}$ instead of ${\bf v}$, consider
\begin{equation*}
{\bf W}(\eta) := r(\tau(\eta))^{\Lambda_\alpha} {\bf w}(\tau(\eta)).
\end{equation*}
Note that all quantities appeared here are $T_\tau$-periodic in $\tau$-time-scale, while are $T_\eta$-periodic in $\eta$-time-scale.
Now the correspondence (\ref{similar-A-Ag}) implies ${\bf W} \in \ker_{\rm per}\mathcal{L}_{A}^{(2)}(\lambda)$, namely
\begin{equation*}
\frac{d{\bf W}}{d\eta} = \{A(\eta) - r(\tau(\eta))^k \lambda I \} {\bf W}.
\end{equation*}
Now let
\begin{equation}
\label{W-V}
\hat {\bf V}(\eta) := e^{\lambda \tau(\eta)}{\bf W}(\eta)
\end{equation}
Then we can construct a solution in $\eta$-time-scale as follows:
\begin{align*}
\frac{d\hat {\bf V}}{d\eta} &= \lambda \frac{d\tau}{d\eta}e^{\lambda \tau(\eta)}{\bf W} + e^{\lambda \tau(\eta)}\frac{d{\bf W}}{d\eta} \\
	&=  \lambda  r(\tau(\eta))^k e^{\lambda \tau(\eta)}{\bf W} + e^{\lambda \tau(\eta)} \{A(\eta) - r(\tau(\eta))^k \lambda I \} {\bf W}\\
	&= e^{\lambda \tau(\eta)} \left[ \lambda  r(\tau(\eta))^k {\bf W} +  \{A(\eta) - r(\tau(\eta))^k \lambda I \} {\bf W} \right] \\
	&= A(\eta) \hat {\bf V}.
\end{align*}
Furthermore, we have
\begin{equation*}
\hat {\bf V}(\eta + T_\eta) = e^{\lambda \tau(\eta + T_\eta)} {\bf W}(\eta + T_\eta) = e^{\lambda T_\tau} e^{\lambda \tau(\eta)} {\bf W}(\eta + T_\eta) = e^{\lambda T_\tau} e^{\lambda \tau(\eta)} {\bf W}(\eta) = e^{\lambda T_\tau} \hat {\bf V}(\eta)
\end{equation*}
Namely, ${\bf v}(\tau)$ in (\ref{w-v}) and $\hat V(\eta)$ in (\ref{W-V}) have {\em identical} multipliers.
%\par

\begin{prop}
Let $M_A$ and $M_{A,g}$ be the monodromy matrices of $A(\eta)$ with the period $T_\eta$ and of $A_g(\tau)$ with the period $T_\tau$, respectively.
Then we have 
\begin{equation*}
M_A = r(0)^{\Lambda_\alpha} M_{A,g} r(0)^{-\Lambda_\alpha},
\end{equation*}
namely two monodromy matrices are similar.
\end{prop}

\begin{proof}
Let $\Phi_A(\eta)$ and $\Phi_{A,g}(\tau)$ be the fundamental matrices normalized by $\Phi_A(0)=\Phi_{A,g}(0)=I$. 
The transformation
\begin{equation*}
{\bf v}(\tau)=r(\tau)^{-\Lambda_\alpha}V(\eta(\tau))
\end{equation*}
yields
\begin{equation*}
\Phi_{A,g}(\tau)
=
r(\tau)^{-\Lambda_\alpha}
\Phi_A(\eta(\tau))
r(0)^{\Lambda_\alpha}.
\end{equation*}
Because $r(T_\tau)=r(0)$ and $\eta(T_\tau)=T_\eta$, we have
\begin{equation*}
M_{A,g}
=
r(0)^{-\Lambda_\alpha}M_A r(0)^{\Lambda_\alpha},
\end{equation*}
which proves the assertion.
%We have
%\begin{align*}
%M_{A,g}{\bf v}(\tau) &= \mu {\bf v}(\tau) = \mu e^{\lambda \tau}{\bf w}(\tau) = \mu e^{\lambda \tau} r(\tau(\eta))^{-\Lambda_\alpha}{\bf W}(\tau(\eta))	\\
%	&= \mu e^{\lambda \tau} r(\tau(\eta))^{-\Lambda_\alpha}e^{-\lambda \tau(\eta)}\hat {\bf V}(\eta) =  \mu  r(\tau(\eta))^{-\Lambda_\alpha}\hat {\bf V}(\eta) =  r(\tau(\eta))^{-\Lambda_\alpha} M_A \hat {\bf V}(\eta),
%\end{align*}
%while
%\begin{equation*}
%{\bf v}(\tau(\eta)) = e^{\lambda \tau(\eta)} {\bf w}(\tau(\eta)) = e^{\lambda \tau(\eta)}  r(\tau(\eta))^{-\Lambda_\alpha}{\bf W}(\eta) = e^{\lambda \tau(\eta)}  r(\tau(\eta))^{-\Lambda_\alpha} e^{-\lambda \tau(\eta)} \hat {\bf V}(\eta) = r(\tau(\eta))^{-\Lambda_\alpha} \hat {\bf V}(\eta).
%\end{equation*}
%Therefore we have
%\begin{equation*}
%{\bf v}(\tau(\eta)) = r(\tau(\eta))^{-\Lambda_\alpha} \hat {\bf V}(\eta)
%\end{equation*}
%and
%\begin{align*}
%M_{A,g}r(\tau(\eta))^{-\Lambda_\alpha} \hat {\bf V}(\eta) =  r(\tau(\eta))^{-\Lambda_\alpha} M_A \hat {\bf V}(\eta).
%\end{align*}
%As a summary, we have
%\begin{align*}
% r(\tau(\eta))^{\Lambda_\alpha} M_{A,g}r(\tau(\eta))^{-\Lambda_\alpha} = M_A.
%\end{align*}
\end{proof}

Here we shall consider the correspondence of characteristic multipliers of $\gamma_\ast$ associated with $D_{\bf x}g({\bf x}_\ast(\tau))$ to those of $\gamma_0$ associated with $A(\eta)$.
Using (\ref{identity-Dg-Ag}), we have
\begin{equation*}
\mathcal{L}_g(\lambda) {\bf w} = \mathcal{L}_{A, g, \mathcal{E}}(\lambda) {\bf w}\quad \text{for all}\quad {\bf w}\in \Gamma_{\rm per}(T\mathcal{E})
\end{equation*}
and hence the following are equivalent:
\begin{itemize}
\item $\ker_{\rm per}\mathcal{L}_{g}(\lambda) \equiv \ker \mathcal{L}_{g}(\lambda)\cap \Gamma_{\rm per}(T\mathcal{E})\not = \{0\}$.
\item $\ker_{\rm per}\mathcal{L}_{A, g, \mathcal{E}}(\lambda) \equiv \ker \mathcal{L}_{A,g,\mathcal{E}}(\lambda)\cap \Gamma_{\rm per}(T\mathcal{E}) \not = \{0\}$.
\end{itemize}

Because $T\mathcal{E}$ is $D_{\bf x}g$-invariant, we also have (\ref{identity-Dg-Ag}) from Proposition \ref{prop-Dg-I-P}.
In this case, letting ${\bf v}(\tau) = e^{\lambda \tau}{\bf w}(\tau)$, we have
\begin{align*}
0 &= \{D_{\bf x}g({\bf x}_\ast(\tau)) -\lambda I \}{\bf w}(\tau) - \dot {\bf w}(\tau) = \{D_{\bf x}g({\bf x}_\ast(\tau)) -\lambda I \} (e^{-\lambda \tau}{\bf v}(\tau)) - \frac{d}{d\tau}(e^{-\lambda \tau}{\bf v}(\tau)) \\
%%%%%
	&= \{D_{\bf x}g({\bf x}_\ast(\tau)) -\lambda I \} (e^{-\lambda \tau}{\bf v}(\tau)) + \lambda e^{-\lambda \tau}{\bf v}(\tau) - e^{-\lambda \tau}\dot {\bf v}(\tau) \\
%%%%%
	&= D_{\bf x}g({\bf x}_\ast(\tau)) (e^{-\lambda \tau}{\bf v}(\tau)) - e^{-\lambda \tau}\dot {\bf v}(\tau).
%%%%%%%
\end{align*}
In particular, the function ${\bf v}(\tau)$ admits a characteristic multiplier $\mu = e^{\lambda T_\tau}$ as a solution to
\begin{equation*}
\dot {\bf v}(\tau) = D_{\bf x}g({\bf x}_\ast(\tau)) {\bf v}(\tau).
\end{equation*}
By the invariance, we also have ${\bf v}\in \Gamma(T\mathcal{E})$, implying that $\mu$ is a multiplier of $D_{\bf x}g({\bf x}_\ast)|_{T\mathcal{E}}$.
As a summary, the following statements are equivalent.
\begin{itemize}
\item $\ker_{\rm per}\mathcal{L}_{g}(\lambda)\not = \{0\}$,
\item $\ker_{\rm per}\mathcal{L}_{A, g, \mathcal{E}}(\lambda)\not = \{0\}$,
\item $\mu = e^{\lambda T_\tau}$ is a characteristic multiplier of $D_{\bf x}g({\bf x}_\ast)|_{T\mathcal{E}}$.
\end{itemize}
Next we lift the above correspondence to multipliers among $A$ and $D_{\bf x}g|_{T\mathcal{E}}$.
A reasonable approach is to compare them {\em on the quotient bundles} reducing the special invariant subbundles, instead of treating on the whole bundles.
Let 
\begin{equation*}
\gamma_0 :=\{ {\bf Y}_0(\eta)\mid \eta\in\mathbb{R}\},\quad
\gamma_\ast :=\{{\bf x}_\ast(\tau)\mid \tau\in\mathbb{R}\}
\end{equation*}
and consider the following bundles:
\begin{equation*}
\mathcal{Q}_{\bf Y} :=(\gamma_0\times\mathbb R^n)/R_{\bf Y}^{(2)},\quad
\mathcal{Q}_{\bf x} := T\mathcal{E}|_{\gamma_\ast} /T\gamma_\ast.
\end{equation*}
%を考える。ここで
%\begin{equation*}
%R_{\bf Y}^{(2)}(\eta)= {\rm span}\{\dot Y_0(\eta),\Lambda_\alpha Y_0(\eta)\},
%\quad
%T\gamma_\ast(\tau)={\rm span}\{\dot {\bf x}_\ast(\tau)\}
%\end{equation*}
%である。
Under the time-scale transformation $\tau=\tau(\eta)$, equivalently $\eta = \eta(\tau)$, introduce the following map between quotient bundles:
\begin{equation*}
\mathcal{S}_\eta: \mathcal{Q}_{\bf x}(\tau(\eta))\to \mathcal{Q}_{\bf Y}(\eta),
\quad
\mathcal{S}_\eta[\xi]:=[r(\tau(\eta))^{\Lambda_\alpha}\xi].
\end{equation*}
This map is well-defined. Indeed, when we replace $\xi$, the representative of $[\xi]$, with $\xi+a\dot {\bf x}_\ast$, we have
\begin{equation*}
r^{\Lambda_\alpha} \left(\xi+a\frac{d{\bf x}_\ast}d\tau \right) - r^{\Lambda_\alpha}\xi = a r^{\Lambda_\alpha}\frac{d{\bf x}_\ast}{d\tau}.
\end{equation*}
Now we know that
\begin{equation*}
r^{\Lambda_\alpha}\dot {\bf x}_\ast = r^{-k}\frac{d{\bf Y}_0}{d\eta} - \frac{1}{r}\frac{dr}{d\tau}\Lambda_\alpha {\bf Y}_0 \in R_{\bf Y}^{(2)}(\eta),
\end{equation*}
which implies that $S_\eta [\xi+a\dot {\bf x}_\ast] = S_\eta [\xi]$ in $\mathcal{Q}_Y$ and shows the well-definedness of $\mathcal{S}_\eta$.
%%%%
Moreover, from $r(\tau)$ being $T_\tau$-periodic and $\tau(\eta+T_\eta)=\tau(\eta)+T_\tau$, we have
\begin{equation*}
\mathcal{S}_{\eta+T_\eta}=\mathcal{S}_\eta
\end{equation*}
As a result, if a solution $q_{\bf x}(\tau)$ to the linear system restricted to $\mathcal{Q}_{\bf x}$ satisfies
\begin{equation*}
[q_{\bf x}](\tau+T_\tau)=\mu [q_{\bf x}](\tau),
\end{equation*}
then the quantity
\begin{equation*}
[q_{\bf Y}](\eta):=\mathcal{S}_\eta [q_{\bf x}](\tau(\eta))
\end{equation*}
satisfies
\begin{equation*}
[q_{\bf Y}](\eta+T_\eta)=\mu [q_{\bf Y}](\eta).
\end{equation*}
In other words, the scaling on quotient bundles {\em preserves the characteristic multipliers}, which is quite different from stationary cases \cite{asym2}.
Furthermore, the following map defines a bundle isomorphism between two quotient bundles:
\begin{equation*}
\mathcal{T}_\eta: \mathcal{Q}_{\bf Y}(\eta)\to \mathcal{Q}_{\bf x}(\tau(\eta)),
\quad
\mathcal{T}_\eta [\zeta]:=[ (I-P_\ast(\tau(\eta))) p_\alpha({\bf Y}_0(\eta))^{-\Lambda_\alpha}\zeta].
\end{equation*}
The map $\mathcal{T}_\eta$ is also well-defined.
Indeed, if we replace the representative $\zeta$ of $[\zeta]\in \mathcal{Q}_{\bf Y}(\eta)$ with
\begin{equation*}
\zeta + a \frac{d{\bf Y}_0}{d\eta} + b \Lambda_\alpha {\bf Y}_0,
\end{equation*}
then we have
\begin{align*}
r^{-\Lambda_\alpha} \frac{d{\bf Y}_0}{d\eta} &= r^k \frac{d{\bf x}_\ast}{d\tau} + r^k \frac{1}{r}\frac{dr}{d\tau} \Lambda_\alpha {\bf x}_\ast,\quad 
%%%
r^{-\Lambda_\alpha} \Lambda_\alpha {\bf Y}_0 = \Lambda_\alpha {\bf x}_\ast,
\end{align*}
and hence
\begin{equation*}
r^{-\Lambda_\alpha} \left( a \frac{d{\bf Y}_0}{d\eta} + b \Lambda_\alpha {\bf Y}_0\right) = a r^k \frac{d{\bf x}_\ast}{d\tau} + \left( a r^k \frac{1}{r}\frac{dr}{d\tau} + b\right) \Lambda_\alpha {\bf x}_\ast.
\end{equation*}
Acting the projection $(I - P_\ast(\tau(\eta)))$ to both sides, the $\Lambda_\alpha {\bf x}_\ast$-terms vanish, and we have
\begin{equation*}
(I - P_\ast(\tau(\eta))) \left\{ r^{-\Lambda_\alpha} \left( a \frac{d{\bf Y}_0}{d\eta} + b \Lambda_\alpha {\bf Y}_0\right)\right\} = a r^k \frac{d{\bf x}_\ast}{d\tau}.
\end{equation*}
Because the right-hand side is an element of $T\gamma_\ast$, it vanishes on the quotient bundle $\mathcal{Q}_{\bf x}$.
As a consequence, we have the following {\em correspondence}.
\begin{thm}
Two quotient bundles
\begin{equation*}
\mathcal{Q}_{\bf Y} =(\gamma_0\times\mathbb R^n)/R_{\bf Y}^{(2)},\quad
\mathcal{Q}_{\bf x} = T\mathcal{E}|_{\gamma_\ast} /T\gamma_\ast
\end{equation*}
are isomorphic via the following bundle isomorphism:
\begin{equation*}
\mathcal{S}_\eta: \mathcal{Q}_{\bf x}(\tau(\eta))\to \mathcal{Q}_{\bf Y}(\eta),
\quad
\mathcal{S}_\eta[\xi]:=[r(\tau(\eta))^{\Lambda_\alpha}\xi],
\end{equation*}
whose inverse is given by
\begin{equation*}
\mathcal{S}_\eta^{-1} = \mathcal{T}_\eta: \mathcal{Q}_{\bf Y}(\eta)\to \mathcal{Q}_{\bf x}(\tau(\eta)),
\quad
\mathcal{T}_\eta[\zeta] :=[ (I-P_\ast(\tau(\eta))) p_\alpha({\bf Y}_0(\eta))^{-\Lambda_\alpha}\zeta].
\end{equation*}
\end{thm}

\begin{proof}
We have already seen that both $\mathcal{S}_\eta$ and $\mathcal{T}_\eta$ are well-defined.
Here we only verify that they are mutually inverted.
For $\xi \in T_{{\bf x}_\ast(\tau)}\mathcal{E}$, apply $\mathcal{T}_\eta$ to $\mathcal{S}_\eta [\xi] = [r^{\Lambda_\alpha}\xi]$. 
Then
\begin{equation*}
\mathcal{T}_\eta \mathcal{S}_\eta [\xi] = [ (I-P_\ast(\tau(\eta))) p_\alpha({\bf Y}_0(\eta))^{-\Lambda_\alpha} r(\tau(\eta))^{\Lambda_\alpha}\xi] =  [ (I-P_\ast(\tau(\eta))) \xi] = [\xi].
\end{equation*}
On the other hand, for $[\zeta]\in \mathcal{Q}_{\bf Y}$, we have
\begin{equation*}
\mathcal{S}_\eta \mathcal{T}_\eta [\zeta] = [r(\tau(\eta))^{\Lambda_\alpha} (I-P_\ast(\tau(\eta))) p_\alpha({\bf Y}_0(\eta))^{-\Lambda_\alpha} \zeta].
\end{equation*}
The difference of this element from $[\zeta]$ itself is
\begin{equation*}
\zeta - r(\tau(\eta))^{\Lambda_\alpha} (I-P_\ast(\tau(\eta))) p_\alpha({\bf Y}_0(\eta))^{-\Lambda_\alpha} \zeta = r(\tau(\eta))^{\Lambda_\alpha} P_\ast(\tau(\eta)) p_\alpha({\bf Y}_0(\eta))^{-\Lambda_\alpha} \zeta.
\end{equation*}
More precisely,
\begin{equation*}
r^{\Lambda_\alpha} P_\ast r^{-\Lambda_\alpha}\zeta = r^{\Lambda_\alpha} \Lambda_\alpha {\bf x}_\ast ({\bf n}^T r^{-\Lambda_\alpha} \zeta) = \Lambda_\alpha r^{\Lambda_\alpha} {\bf x}_\ast ({\bf n}^T r^{-\Lambda_\alpha} \zeta) = ({\bf n}^T r^{-\Lambda_\alpha} \zeta) \Lambda_\alpha {\bf Y}_0.
\end{equation*}
The difference belongs to the subbundle $R_{\bf Y}^{(2)}$ indicating that it vanishes on the quotient bundle $\mathcal{Q}_{\bf Y}$.
The proof of $\mathcal{T}_\eta = \mathcal{S}_\eta^{-1}$ is therefore completed.
\end{proof}
%\par
%\bigskip
Denote $\overline M_A$ and $\overline M_g$ be the monodromy matrices of $A(\eta)$ and $D_{\bf x}g({\bf x}_\ast(\tau))|_{T\mathcal{E}}$, induced on $\mathcal{Q}_{\bf Y}$ and $\mathcal{Q}_{\bf x}$, respectively.
Then we have
\begin{equation*}
\overline M_A = 
\mathcal{S}_0 \overline M_g \mathcal{S}_0^{-1}
\end{equation*}
In particular, the characteristic multipliers of the associated linear systems {\em outside} invariant subbundles $R_{\bf Y}^{(2)}$ and $R_{\bf x}^{(2)}$ coincide.

\begin{thm}[Algebraic-Dynamical Correspondence of Linearized Structure]
\label{thm-correspondence-C1}
Suppose that $\gamma_\ast = \{{\bf x}_\ast(\tau)\} \in P_{\mathcal{E}}(g)$ corresponds to a periodic solution ${\bf Y}_\ast(\eta) \equiv {\bf Y}_0(\eta)$ following the correspondence in Theorem \ref{thm-correspondence-C0}.
Also suppose that Assumption \ref{ass-diff-QH} holds.
Then the associated characteristic multipliers and vectors on the subbundles correspond to each other according to Tables \ref{table-eigen1} and \ref{table-eigen2}.
\end{thm}

\begin{table}[ht]\em
\centering
%\resizebox{\textwidth}{!}
{
\begin{tabular}{cccc}
\hline
  & $A(\eta)$ & $D_{\bf x}g({\bf x}_0(\tau))$\\
\hline\\[-2mm]
Multipliers for $R_\ast^{(2)}$ & $1$, $\exp\left( T_\eta \right)$ & $1$, $\exp\left(-  \int_0^{T_\eta} \frac{G({\bf Y}_0 (\eta))}{r(\eta)^{2c} } d\eta \right) $\\ [1mm]
%%%
Bases for $R_\ast^{(2)}$ & $\frac{d{\bf Y}_0}{d\eta}(\eta)$, $\Lambda_\alpha {\bf Y}_0(\eta)$ & $\frac{d}{d\tau}\{ r(\eta)^{-\Lambda_\alpha}{\bf Y}_0(\eta)\}$, $\Lambda_\alpha r(\eta)^{-\Lambda_\alpha}{\bf Y}_0(\eta)$ \\  [1mm]
%%%
Multipliers for quotient & $\mu$ & $\mu$ \\  [1mm]
Bases for quotient & $[\zeta]$ & $[ (I-P_\ast(\tau(\eta))) r(\tau(\eta))^{-\Lambda_\alpha}\zeta]$ \\  [1mm]
\hline 
\end{tabular}%
}
\caption{Correspondence of eigenstructures from $A(\eta)$ to $D_{\bf x}g({\bf x}_0(\tau))$}
\flushleft
The bundle $R_\ast^{(2)}$ denotes $R_{\bf Y}^{(2)}$ or $R_{\bf x}^{(2)}$ correspondingly.
The function $r(\eta) = r(\tau (\eta))$ is $p_\alpha({\bf Y}_0(\eta))$. 
Once a periodic solution ${\bf Y}_0(\eta)$ of the balance law and characteristic multipliers of $A(\eta)$ are given, corresponding periodic solution on the horizon ${\bf x}_0(\tau(\eta))$ and all characteristic multipliers of $D_{\bf x}g({\bf x}_0(\tau))$ are constructed by the rule on the table.
\label{table-eigen1}
\end{table}%

\begin{table}[ht]\em
\centering
%\resizebox{\textwidth}{!}
{
\begin{tabular}{cccc}
\hline
  & $D_{\bf x}g({\bf x}_\ast(\tau))$  & $A(\eta)$\\
\hline\\[-2mm]
Multipliers for $R_\ast^{(2)}$ & $1$, $\exp\left(- \int_0^{T_{\tau}}C_\ast(\tau)d\tau \right)$ & $1$, $\exp\left( T_\eta \right)$ \\ [1mm]
%%%
Bases for $R_\ast^{(2)}$ & $\frac{d{\bf x}_\ast}{d\tau}(\tau)$, $\Lambda_\alpha {\bf x}_\ast(\tau)$ & $\frac{d}{d\eta}\{ r(\tau(\eta))^{\Lambda_\alpha}{\bf x}_\ast(\tau(\eta))\}$, $\Lambda_\alpha r(\tau(\eta))^{\Lambda_\alpha}{\bf x}_\ast(\tau(\eta))$ \\  [1mm]
%%%
Multipliers for quotient & $\mu$ & $\mu$ \\  [1mm]
Bases for quotient & $[\xi]$ & $[r(\tau(\eta))^{\Lambda_\alpha} \xi]$ \\  [1mm]
\hline 
\end{tabular}%
}
\caption{Correspondence of eigenstructures from $D_{\bf x}g({\bf x}_\ast(\tau))$ to $A(\eta)$}
\flushleft
The bundle $R_\ast^{(2)}$ denotes $R_{\bf x}^{(2)}$ or $R_{\bf Y}^{(2)}$ correspondingly.
The function $r(\tau) = r(\tau(\eta))$ is given in Lemma \ref{lem-r} provided $\int_0^{T_{\tau}}C_\ast(\tau)d\tau > 0$.
Once a periodic solution ${\bf x}_\ast(\tau)$ of the desingularized vector field $g$ and characteristic multipliers of $D_{\bf x}g({\bf x}_\ast(\tau))$ are given, corresponding periodic solution ${\bf Y}_\ast(\eta)$ of the balance law and all characteristic multipliers of $A(\eta)$ are constructed by the rule on the table.
\label{table-eigen2}
\end{table}%

\section{Criterion of Periodic Blow-Ups and Examples}
\label{section-example}

\subsection{Simple criterion of the existence of periodic blow-ups}

Now we are ready to state a simple criterion of the existence of periodic blow-ups without directly using embeddings.
The proof is, like the stationary case \cite{asym2}, combination of blow-up descriptions via dynamics at infinity, in particular hyperbolic nature of invariant sets on the horizon; Theorem \ref{thm:blowup} (cf. \cite{Mat2018, Mat2025_NHIM}), as well as the algebraic-dynamical correspondence of linearized structures, Theorem \ref{thm-correspondence-C1}.

\begin{thm}[Criterion of existence of periodic blow-ups]
\label{thm-existence-blow-up}
Let $f \in AQH_{n}(\alpha, k)$ satisfying Assumption \ref{ass-diff-QH}, and ${\bf Y}_0(\eta)$ be a periodic solution to the balance law (\ref{dfn-balance}) with $p_\alpha({\bf Y}_0(\eta)) > 0$ for all $\eta$.
Assume that the corresponding characteristic multipliers ${\rm Ch}({\bf Y}_0)$ consist of $1$ as the simple multiplier and 
\begin{equation*}
\left\{ {\rm Ch}({\bf Y}_0)\setminus \{1\} \right\} \cap S^1 = \emptyset,
\end{equation*}
where $S^1$ denotes the unit circle in $\mathbb{C}$.
Then (\ref{ODE-original}) possesses a blow-up solution ${\bf y}(t)$ with the asymptotic behavior $ {\bf y}(t) \sim {\bf Y}_0(-c'\ln \theta(t)) \theta(t)^{-\frac{1}{k}\Lambda_\alpha}$ as $t\to t_{\max} < \infty$ for some constant $c' >0$.
\end{thm}

Through the above theorem, only the balance law is sufficient to consider the existence of type-I oscillatory blow-ups.
In particular, no explicit construction of a compactification is necessary for applying the criterion and the blow-up ansatz is justified, as far as type-I behavior is considered.
Additionally, also like the stationary case, the correspondence theorem shows the gap of stability information among two systems, which will provide a caution to consider the dynamical nature of blow-ups {\em based on specific ansatz}.

\begin{thm}[Stability gap]
\label{thm-stability}
Let $f \in AQH_{n}(\alpha, k)$ satisfying Assumption \ref{ass-diff-QH}.
Let $\gamma_\ast$ be a hyperbolic periodic orbit on the horizon for the desingularized vector field $g$ associated with $f$ such that $W^s_{\rm loc}(\gamma_\ast; g)\cap \mathcal{D}\not = \emptyset$, and ${\bf Y}_0(\eta)$ be the corresponding periodic solution of the balance law (\ref{dfn-balance}) with $p_\alpha({\bf Y}_0(\eta)) > 0$ for any $\eta$.
If
\begin{align*}
m &:= \dim W^s_{\rm loc}(\gamma_\ast; g),\quad m_A = \sharp \{ \mu\in {\rm Ch}({\bf Y}_0) \mid |\mu| \leq 1 \},
\end{align*}
then we have $m = m_A+1$, where characteristic multipliers are counted with algebraic multiplicity.
\end{thm}

In the successive subsections, two examples to confirm the availability of our results are collected.
These examples come from preceding results.
\begin{rem}[Numerical computation of characteristic multipliers]
In the computations below, periodic solutions are numerically computed by a single-shooting Newton method supplemented with a phase condition. 
Characteristic multipliers are then computed in the standard manner by simultaneously integrating the variational equation along the periodic solution (see, e.g., \cite{K2004} for the standard Floquet-theoretic background).
More precisely, for a periodic solution ${\bf Y}_0(\eta)$ of ${\bf Y}'= F({\bf Y})$ with period $T_\eta$, we solve
\begin{equation*}
\frac{d\Phi}{d\eta} = D_{\bf Y}F({\bf Y}_0(\eta)) \Phi,\quad \Phi(0) = I.
\end{equation*}
The matrix $\Phi(T_\eta)$ is the monodromy matrix, and its eigenvalues are taken as the characteristic multipliers. 
Periodic orbits and associated monodromy matrices are numerically computed by the shooting procedure described above.
\end{rem}

\subsection{Example 1: Li\'{e}nard-type equation}
The first example is the two-dimensional system
\begin{equation}
\label{Lienard}
\begin{cases}
\dot x = y, & \\
\dot y = -(\epsilon x^m + \sum_{k=0}^{m-1}a_k x^k) - y(x^n + \sum_{k=0}^{n-1}b_k x^k), & 
\end{cases}
\end{equation}
which is referred to as the {\em Li\'{e}nard equation}.
According to e.g. \cite{DH1999}, {\em periodic orbits at infinity} can be seen in (\ref{Lienard}) with appropriate degree, where $\epsilon = \pm 1$ if $m\not = 2n+1$, and $\epsilon \in \mathbb{R}\setminus \{0\}$ if $m=2n+1$.
%%%
\par
First note that the system (\ref{Lienard}) with $m=2n+1$ is $AQH_2(\alpha, n)$ with $\alpha = (1,n+1)$.
Arguments in \cite{Mat2018} show that the whole horizon is a hyperbolic repelling periodic orbit.
To confirm this observation in simple situation, consider the following system
\begin{equation}
\label{Lienard-simple}
\begin{cases}
y_1' = y_2, & \\
y_2' = -y_1^{2n+1} - y_1^n y_2. & 
\end{cases}
\end{equation}
In \cite{Mat2018}, the $(1,n+1)$-{\em quasipolar compactification coordinate} $(r,\theta)$ given by
\begin{equation*}
y_1 = \frac{{\rm Cs}\theta}{r},\quad y_2 = \frac{{\rm Sn}\theta}{r^{n+1}}\quad\text{ with }\quad {\rm Cs}^{2n+2}\theta + (n+1){\rm Sn}^2 \theta = 1
\end{equation*}
was applied to investigating type-I periodic blow-up with an appropriate time direction {\em when $n$ is even}.
%\par
\par
\bigskip
Here we shall investigate such a periodic blow-up from the viewpoint of the correspondence.
We apply the time-reversing to correspond to results in \cite{Mat2018}:
\begin{equation}
\label{Lienard-simple-reverse}
\begin{cases}
y_1' = -y_2, & \\
y_2' = y_1^{2n+1} + y_1^n y_2. & 
\end{cases}
\end{equation}
Note that (\ref{Lienard-simple-reverse}) belongs to $QH_2(\alpha, n)$.
Introducing
\begin{equation*}
y_1 = \frac{x_1}{1 - p_\alpha({\bf x})^{2(n+1)}},\quad y_2 = \frac{x_2}{ (1 - p_\alpha({\bf x})^{2(n+1)})^{n+1} },\quad p_\alpha({\bf x})^{2(n+1)} = x_1^{2(n+1)} + x_2^2,
\end{equation*}
we have the following desingularized vector field:
\begin{align}
\label{Lienard-simple-reverse-desing}
\dot{\bf x} &\equiv \frac{d{\bf x}}{d\tau} = \frac{1}{2(n+1)}\{1 + (2n+1)p_\alpha({\bf x})^{2(n+1)}\}\tilde f({\bf x}) - G({\bf x}) \Lambda_\alpha {\bf x},\\
%%%
\notag
\tilde f({\bf x}) &= (\tilde f_1({\bf x}), \tilde f_2({\bf x}))^T,\quad \tilde f_1({\bf x}) = -x_2,\quad \tilde f_2({\bf x}) = x_1^{2n+1} + x_1^n x_2,\\
\notag
G({\bf x}) &\equiv \sum_{j=1}^2 \frac{x_j^{2\beta_j - 1}}{\alpha_j} \tilde f_j({\bf x}) = -x_1^{2n+1}x_2 + \frac{x_2}{n+1}\{x_1^{2n+1} + x_1^n x_2\}.
\end{align}
The corresponding balance law is
\begin{equation}
\label{Lienard-simple-reverse-balance}
\frac{dY_1}{d\eta} = -\frac{1}{n}Y_1 - Y_2,\quad \frac{dY_2}{d\eta} = -\frac{n+1}{n}Y_2 + \{Y_1^{2n+1} + Y_1^n Y_2\}.
\end{equation}

Numerical computations show the following observations, which confirm the validity of our theoretical arguments, and an approach to provide the existence of blow-ups without compactifications.
\begin{observation}
Consider (\ref{Lienard-simple}) with the backward-time direction, namely (\ref{Lienard-simple-reverse}) with $n=2$.
Both systems (\ref{Lienard-simple-reverse-desing}) and (\ref{Lienard-simple-reverse-balance}) admit periodic solutions, in which the orbit for (\ref{Lienard-simple-reverse-desing}) is located on the horizon.
Associated quantities are collected in Table \ref{table-Lienard}.
In particular, the following identity is confirmed:
\begin{align*}
T_\eta &:= 1.262967766799167,\quad T_\tau := 6.740686419615101,\\
\int_0^{T_\tau} & G({\bf x}_\ast(\tau))d\tau \approx 0.6314838833996594 \approx \frac{1}{2}T_\eta,
\end{align*}
confirming (\ref{period-eta}) with $k=n=2$.
Moreover, we easily see that the characteristic multiplier of (\ref{Lienard-simple-reverse-balance}) other than $1$ is $\exp(T_\eta)$.
Through the periodic trajectory $\{{\bf x}_\ast(\tau)\}\subset \mathcal{E}$, we also see that the multiplier of (\ref{Lienard-simple-reverse-desing}) other than $1$ is $\exp(-\int_0^{T_\tau} G({\bf x}_\ast(\tau))d\tau)$.
\end{observation}
The concrete profile of the corresponding periodic solution of (\ref{Lienard-simple-reverse-balance}) is drawn in Figure \ref{fig-Lienard}.

\begin{table}[ht]\em
\centering
%\resizebox{\textwidth}{!}
{
\begin{tabular}{ccc}
\hline
  & Balance Law (\ref{Lienard-simple-reverse-balance}) & Dynamics at Infinity (\ref{Lienard-simple-reverse-desing}) \\
\hline\\[-2mm]
Period & $1.262967766799167$ & $6.740686419615101$ \\ [1mm]
%%%
Characteristic Multipliers ($\not = 1$) & $3.535899656899745$ & $0.5318020829448200$ \\  [1mm]
%%%
\hline 
\end{tabular}%
}
\caption{Quantities characterizing periodic profiles of (\ref{Lienard-simple-reverse-balance}) and (\ref{Lienard-simple-reverse-desing}) expressing blow-up}
%\flushleft
\label{table-Lienard}
\end{table}%

\begin{figure}[htbp]\em
\centering
\includegraphics[width=7cm]{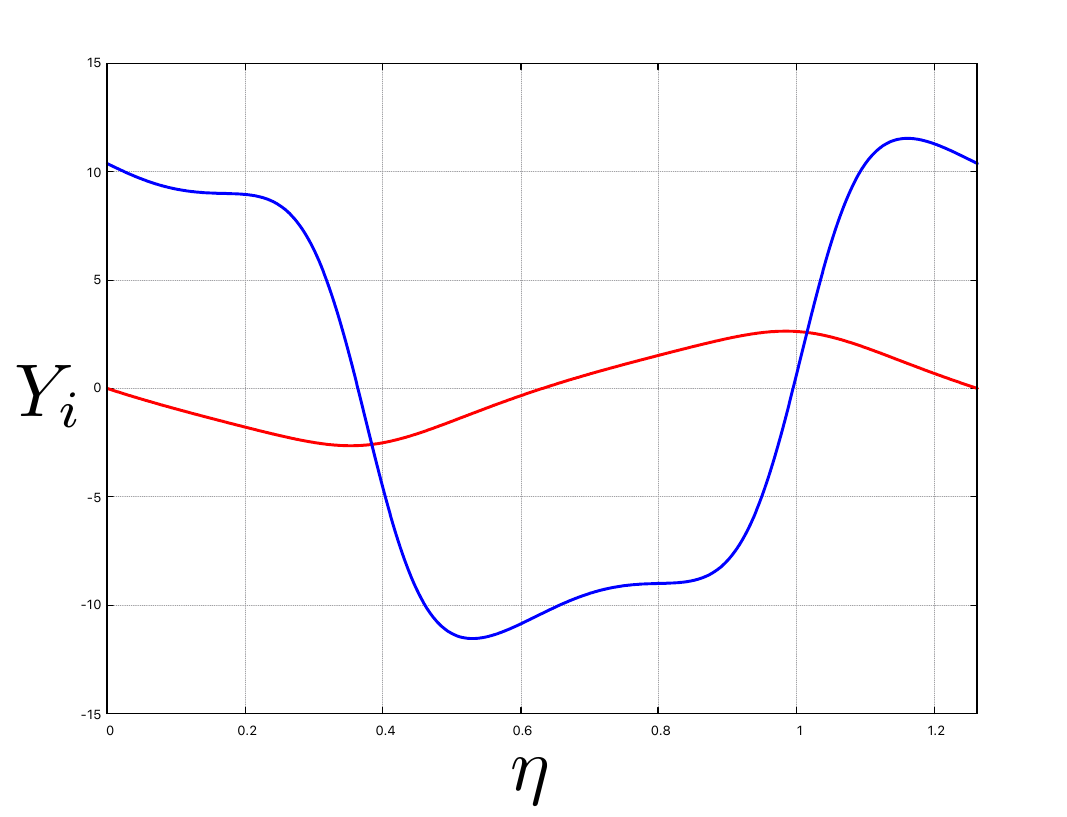}
\caption{Periodic solution of (\ref{Lienard-simple-reverse-balance})}
\label{fig-Lienard}
One-period profiles are drawn.
Red: $Y_1(\eta)$, Blue: $Y_2(\eta)$.
\end{figure}

\subsection{Example 2: Suspension bridge equation}
The next example is the blow-up profile of solutions of
\begin{equation}
\label{SB}
\frac{d^4u}{dr^4}(r) + a \frac{d^2u}{dr^2}(r) + f(u(r)) = 0,
\end{equation}
where $a \in \mathbb{R}$ and $f$ is a locally Lipschitz function.
This equation serves as a model for pattern formation in many physical, chemical or biological systems, and is often used to investigate localization and spreading of deformation of a strut confined by an elastic foundation (cf. \cite{PT2001}).
In some situations (\ref{SB}) admits oscillatory blow-up solutions whose details are stated in e.g., \cite{BFGK2011, GP2011, GP2013}, while their profiles by means of bounded periodic solutions for appropriately transformed equations are investigated  (\cite{DLP2015} etc.).
Here we revisit such an example again and apply our methodology to it as well as expand the existence result of oscillatory blow-ups.
\par
Now we restrict our attention to $f(u) = u^3$, and rewrite (\ref{SB}) as the system of the first order ODEs:
\begin{equation}
\label{SB-system}
\begin{cases}
u_1' = u_2, \\
u_2' = u_3, \\
u_3' = u_4, \\
u_4' = -a u_3 - u_1^3. \\
\end{cases}
\end{equation}
Then we have the following observation.

\begin{lem}
The vector field of (\ref{SB-system}) belongs to $AQH_4(\alpha, 1)$, where $\alpha = (2,3,4,5)$.
\end{lem}

It is reported in \cite{CLM2018, DLP2015} that, under the ansatz
\begin{equation}
\label{ansatz-SB}
u(r) = (1 - r)^{-2}w(\ln (1-r)^{-1}),\quad t = -\ln(1-r),
\end{equation}
the problem {\em with $a=0$} is reduced to a bounded solution to the following ODE\footnote{
%	footnote
In \cite{CLM2018, DLP2015}, the variable $s = -\ln (1-r)$ is used as a transformed time variable.
On the other hand, we have used $s$ as a scaling parameter in our framework, and hence we shall rewrite $t$ as the transformed time variable in (\ref{SB-DLP}).
Indeed, we will not discuss the details of (\ref{SB-DLP}) except showing computer-assisted results there.
%	footnote: end
}:
\begin{equation}
\label{SB-DLP}
w^{(4)} + 14w^{(3)} + 71 w'' + 154 w' + 120 w + w^3 = 0,\quad ' = \frac{d}{dt},
\end{equation}
and one periodic solution $w(t)$ is validated\footnote{
%	footnote
In the following work \cite{CLM2018}, the associated stable/unstable bundles and local stable manifolds are constructed with computer-assisted proofs.
%	footnote: end
} whose profile is numerically reconstructed and is shown in Figure \ref{fig-SB}-(a).
The period of the validated periodic solution is enclosed by
\begin{equation*}
T \in 1.90809723205_{0663}^{1545}.
\end{equation*}
The associated Floquet {\em exponents} reported in \cite{CLM2018, DLP2015} are $\lambda = 0, -7, -8$, and the corresponding characteristic multipliers are
\begin{equation*}
\begin{pmatrix}
1\\
1.58222\times 10^{-6}\\
2.34742221\times 10^{-7}
\end{pmatrix}.
\end{equation*}
The remaining characteristic multiplier is $\mu_0 \approx 6.7402514547258665$, according to our numerical computation\footnote{
%	footnote
In \cite{CLM2018}, this quantity is not shown, which could be because the authors have focused on {\em stable} nature of periodic orbits in the research.
%	footnote:end
}. 
\par
\bigskip
Keeping this preceding result in mind, we shall investigate a periodic orbit of the balance law associated with (\ref{SB-system}) with characteristics to compare with the above quantities.
In the present case, we have set $a = 1$ as an example but, as we shall see, the setting makes no contribution to blow-up description.
\par
Now the balance law associated with (\ref{SB-system}) is 
\begin{equation}
\label{SB-system-balance}
\frac{d}{d\eta}{\bf U} = -\begin{pmatrix}
2 & 0 & 0 & 0\\
0 & 3 & 0 & 0\\
0 & 0 & 4 & 0\\
0 & 0 & 0 & 5\\
\end{pmatrix}{\bf U} + \begin{pmatrix}
U_2\\
U_3\\
U_4\\
-U_1^3
\end{pmatrix}
\end{equation}

The period is
\begin{equation*}
T_\eta \approx 1.90809723205111159.
\end{equation*}
Corresponding multipliers are
\begin{equation*}
\begin{pmatrix}
6.74025145472546861\\
1\\
1.58222159613619978\times 10^{-6}\\
2.34742216662795037\times 10^{-7}
\end{pmatrix},
\end{equation*}
while the corresponding Floquet exponents (other than $0$) are
\begin{align*}
&T_\eta^{-1}\ln (6.74025145472546861) \approx 1,\\
&T_\eta^{-1}\ln (1.58222159613619978\times 10^{-6}) \approx -7.00000000035,\\
&T_\eta^{-1}\ln (2.34742216662795037\times 10^{-7}) \approx -8.00000000988.
\end{align*}
These observations show good agreement with computer-assisted results in \cite{DLP2015} and imply that {\em the asymptotic  quantitative nature of periodic blow-ups is extracted only by the quasi-homogeneous components}, while concrete solution profiles can be affected by lower-order terms, as shown in Figure \ref{fig-SB}.
Moreover, the remaining multiplier $\mu_0$ from (\ref{SB-DLP}) is in good agreement with the theoretical value $\exp(T_\eta)$. 
\par
\bigskip
Next consider the compactification
\begin{align*}
y_1 &= \frac{x_1}{(1 - p_\alpha({\bf x})^{120})^{2}},\quad y_2 = \frac{x_2}{ (1 - p_\alpha({\bf x})^{120})^{3} },\\
y_3 &= \frac{x_3}{(1 - p_\alpha({\bf x})^{120})^{4}},\quad y_4 = \frac{x_4}{ (1 - p_\alpha({\bf x})^{120})^{5} },\\
p_\alpha({\bf x})^{120} &= x_1^{60} + x_2^{40} + x_3^{30} + x_4^{24},
\end{align*}
where we have chosen $\beta = (30, 20, 15, 12)$ and $c = 60$.
The associated desingularized vector field is:
\begin{align}
\label{SB-system-desing}
\dot{\bf x} &\equiv \frac{d{\bf x}}{d\tau} = \frac{1}{120}\{1 + 119 p_\alpha({\bf x})^{120}\}\tilde f({\bf x}) - G({\bf x}) \Lambda_\alpha {\bf x},\\
%%%
\notag
\tilde f({\bf x}) &= (\tilde f_1({\bf x}), \tilde f_2({\bf x}), \tilde f_3({\bf x}), \tilde f_4({\bf x}))^T,\\
\notag
\tilde f_1({\bf x}) &= x_2,\quad \tilde f_2({\bf x}) = x_3,\quad \tilde f_3({\bf x}) = x_4,\quad \tilde f_4({\bf x}) = - a(1-p_\alpha({\bf x})^{120})^{2} x_3 - x_1^3,\\
\notag
G({\bf x}) &\equiv \sum_{j=1}^4 \frac{x_j^{2\beta_j - 1}}{\alpha_j} \tilde f_j({\bf x}) = \frac{x_1^{59}x_2}{2} + \frac{x_2^{39}x_3}{3} + \frac{x_3^{29}x_4}{4} + \frac{x_4^{23}}{5}\left\{ - a(1-p_\alpha({\bf x})^{120})^{2} x_3 - x_1^3 \right\}.
\end{align}
At a glance, direct (rigorous) investigations of (\ref{SB-system-desing}) will require tedious and cumbersome calculations to obtain meaningful information, such as computer-assisted proofs (e.g., \cite{CLM2018, DLP2015, LMT2023}).
Once we see the correspondence we have constructed, we can connect genuine blow-up nature from the balance law, which looks similar to the approaches in \cite{CLM2018, DLP2015} under several restrictions, with (\ref{SB-system-desing}).
\par
Nevertheless, numerical computations are feasible in a reasonable manner, and hence we shall numerically investigate the nature of periodic solutions to (\ref{SB-system-desing}) on the horizon.
Indeed, {\em the standard $4$-th order Runge-Kutta scheme} provides a periodic orbit of (\ref{SB-system-desing}) in $\mathcal{E}$ with a typical choice of initial points\footnote{
%	footnote
In principle, \lq\lq unstable" orbits would not be captured in a naive way.
This observation, on the other hand, implies that {\em the orbit is stable indeed, unlike observations in \cite{CLM2018, DLP2015}}.
%	footnote: end
}.
The period is
\begin{equation*}
T_\tau \approx 9.1220661214704784,
\end{equation*}
while corresponding multipliers are
\begin{equation*}
\begin{pmatrix}
1\\
0.14836241756777283\\
1.5822215141006154\times 10^{-6}\\
2.3474229458044122\times 10^{-7}
\end{pmatrix} .
\end{equation*}
We see that the above quantities are consistent with our algebraic-dynamical correspondence (Theorem \ref{thm-correspondence-C1}), while
\begin{equation*}
- \ln(0.14836241756777283) \approx 1.90809723121 \approx T_\eta = \int_0^{T_\tau} G({\bf x}_\ast(\tau))d\tau.
\end{equation*}

\begin{observation}
Consider (\ref{SB-system}) with $a = 1$.
Both systems (\ref{SB-system-balance}) and (\ref{SB-system-desing}) admit periodic solutions, in which the orbit for (\ref{SB-system-desing}) is located on the horizon $\{p_\alpha \equiv 1\}$.
In particular, correspondence of characteristic multipliers mentioned in Theorem \ref{thm-correspondence-C1} is confirmed with $k=1$.
Moreover, period and characteristic multipliers of the periodic orbit for (\ref{SB-system-balance}) are (numerically) identical with (\ref{SB-DLP}) under the ansatz (\ref{ansatz-SB}). 
In particular, at least in the $C^0$- and $C^1$-sense, the \lq\lq lower-order term" $-au''$ in (\ref{SB}) makes no contribution to describe periodic blow-ups.
\end{observation}

\begin{figure}[htbp]\em
\centering
\includegraphics[width=15cm]{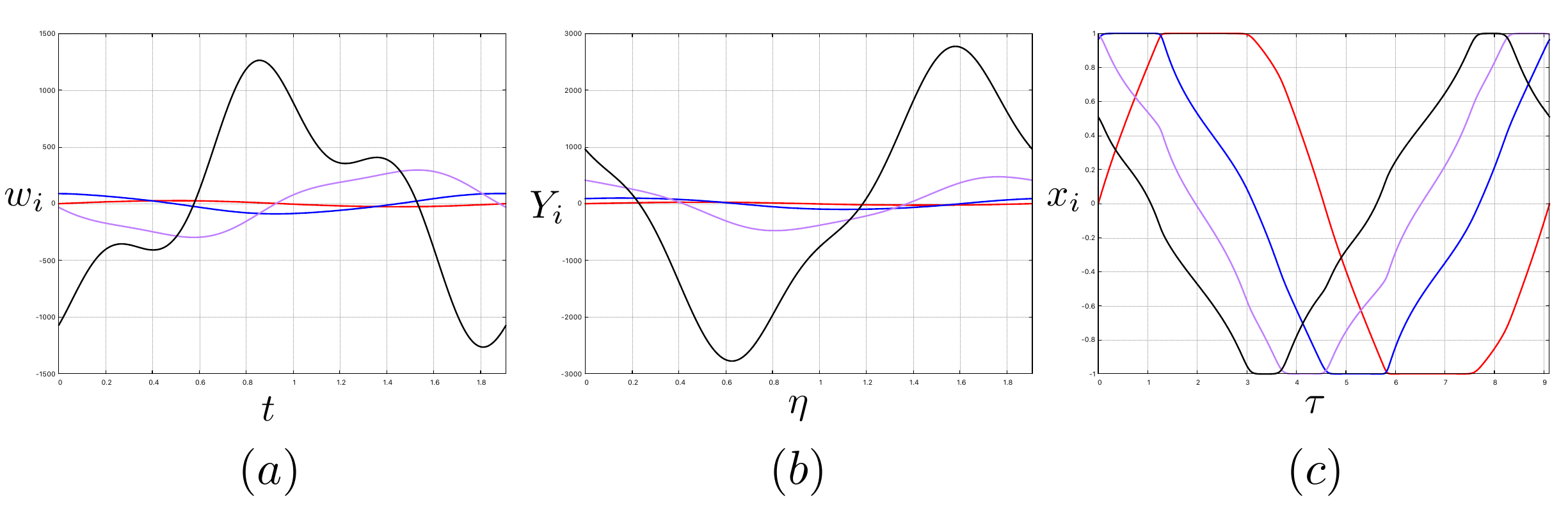}
\caption{Periodic solutions of (\ref{SB-DLP}), (\ref{SB-system-balance}) and (\ref{SB-system-desing}).}
\label{fig-SB}
(a): to (\ref{SB-DLP}), (b): to (\ref{SB-system-balance}), (c): to (\ref{SB-system-desing}).
One-period profiles are drawn.
In each figure, Red: $i=1$, Blue: $i=2$, Purple: $i=3$, Black: $i=4$.
\end{figure}

\section*{Concluding Remarks}
The present paper addresses the correspondence of (type-I) blow-up nature between dynamics at infinity and balance laws.
This investigation originated from {\em stationary} objects in \cite{asym2, M2026_asym_nonaut}, and the present study extends this understanding to {\em oscillatory} blow-up behavior, in particular to {\em periodic} structures associated with blow-up.
Combining approaches in the present arguments with descriptions of blow-ups shadowing more general \lq\lq invariant sets at infinity", such as NHIMs on the horizon \cite{Mat2025_NHIM}, further generalization of the correspondence will be constructed.
\par
Of particular importance in constructing the correspondence is the determination of the function $r(\tau)$ or $r(\tau(\eta))$.
Our observation here is that $r$ is determined by a (non-trivial) positive solution to a linear, scalar differential equation, and the condition to guarantee its existence is characterized by the functional $G({\bf x}(\tau))$.
\par
\bigskip
On the other hand, the (average) positivity of the functional $G({\bf x}(\tau))$ characterizes the hyperbolic behavior in the direction transverse to the horizon.
Indeed, as shown in \cite{asym2} for stationary case and Theorem \ref{thm-correspondence-C1} for periodic case, the distinguished transverse eigenvalue in the stationary case is $-C_\ast \equiv -G({\bf x_\ast})$, while the corresponding characteristic multiplier in the periodic case is $\exp(-\int_0^{T_\tau} C_\ast(\tau) d\tau)$ with $C_\ast(\tau) = G({\bf x}_\ast(\tau))$.
Therefore, {\em positivity of $C_\ast$ in the stationary case and positivity of its period average in the periodic case guarantee transverse attraction toward the horizon} with an appropriate choice of initial points, inducing type-I finite-time blow-ups.
We shall then expect that the quantity $C_\ast$ will provide information of type-I blow-ups in a wide class of evolutionary systems, such as infinite dimensional systems (partial / delay differential equations, for example) or finite-dimensional systems  {\em without admitting smooth invariant manifold structure of \lq\lq invariant sets at infinity"}.
This direction shall be investigated in the near future.

\section*{Acknowledgements}
KM was partially supported by World Premier International Research Center Initiative (WPI), Ministry of Education, Culture, Sports, Science and Technology (MEXT), Japan, the Research Institute for Mathematical Sciences,
an International Joint Usage/Research Center located in Kyoto University, JSPS Grant-in-Aid for Scientific Research (B) (No. JP26K00619) and the financial support from the JSPS Bilateral Program (PHC Sakura).

\bibliographystyle{plain}
\bibliography{blow_up_asymptotic_per}

%\section{Proofs of statements}

\end{document}